\documentclass[12pt, twoside]{article}
\usepackage{amsmath,amsthm,amssymb}
\usepackage{times}
\usepackage{enumerate}

\usepackage[ansinew]{inputenc}
\usepackage{graphicx}
\usepackage{enumerate,latexsym}
\usepackage{latexsym}
\usepackage{amsmath,amssymb}

\newfont{\bb}{msbm10 at 11pt}
\newfont{\bbsmall}{msbm8 at 8pt}

\def\rth{\mathbb{R}^3}
\def\R{\mathbb{R}}
\def\B{\mathbb{B}}
\def\N{\mathbb{N}}
\def\Z{\mathbb{Z}}

\def\Hip{\mathbb{H}}

\def\esf{\mathbb{S}}

\newcommand{\ben}{\begin{enumerate}}
\newcommand{\bit}{\begin{itemize}}
\newcommand{\een}{\end{enumerate}}
\newcommand{\eit}{\end{itemize}}
\newcommand{\wh}{\widehat}
\newcommand{\Int}{\mbox{\rm Int}}
\newcommand{\ds}{\displaystyle}

\newcommand{\wt}{\widetilde}
\newcommand{\HH}{\mathbb{H}}

\newcommand{\ed}{\end{document}}
\newcommand{\ov}{\overline}

\def\a{{\alpha}}

\def\t{{\theta}}

\def\g{{\gamma}}
\def\G{{\Gamma}}
\def\l{{\lambda}}

\def\de{{\delta}}
\def\be{{\beta}}
\def\ve{{\varepsilon}}

\def\cA{\mathcal{A}}

\def\cC{\mathcal{C}}

\def\cH{\mathcal{H}}

\def\cM{\mathcal{M}}

\def\cF{\mathcal{F}}

\def\cD{\mathcal{D}}

\let\8=\infty \let\0=\emptyset

\let\alfa=\alpha

\let\parc=\partial

\def\ep{\varepsilon}

\def\esiz{\langle}
\def\esde{\rangle}

\def\cte.{\mathop{\rm cte.}\nolimits}

\def\div{\mathop{\rm div }\nolimits}

 \def\dim{\mathop{\rm dim }\nolimits}

\def\N{\mathbb{N}}
\def\E{\mathbb{E}}
\def\B{\mathbb{B}}

\def\R{\mathbb{R}}
\def\Z{\mathbb{Z}}

\def\H{\mathbb{H}}
\def\S{\mathbb{S}}

\def\psl{{\rm PSL}(2,\mathbb{R})}

\newtheorem{theorem}{Theorem}[section]
\newtheorem{lemma}[theorem]{Lemma}
\newtheorem{proposition}[theorem]{Proposition}

\newtheorem{remark}[theorem]{Remark}
\newtheorem{corollary}[theorem]{Corollary}
\newtheorem{definition}[theorem]{Definition}
\newtheorem{conjecture}[theorem]{Conjecture}
\newtheorem{assertion}[theorem]{Assertion}

\newcommand{\su}{{\rm SU}(2)}

\renewcommand{\sl}{\wt{\mathrm{SL}}(2,\R)}

\numberwithin{equation}{section}

\def\titlerunning#1{\gdef\titrun{#1}}
\makeatletter
\def\author#1{\gdef\autrun{\def\and{\unskip, }#1}\gdef\@author{#1}}
\def\address#1{{\def\and{\\\hspace*{18pt}}\renewcommand{\thefootnote}{}%
\footnote {#1}}%
\markboth{\autrun}{\titrun}}
\makeatother
\def\email#1{e-mail: #1}
\def\subjclass#1{{\renewcommand{\thefootnote}{}%
\footnote{\emph{Mathematics Subject Classification (2010):} #1}}}
\def\keywords#1{\par\medskip
\noindent\textbf{Keywords.} #1}

\theoremstyle{definition}

\numberwithin{equation}{section}

\begin{document}


\baselineskip=17pt


\titlerunning{Isoperimetric domains in homogeneous three-manifolds}

\title{Isoperimetric domains of large volume in homogeneous three-manifolds}

\author{William H. Meeks III
\and
Pablo Mira
\and
Joaqu\'\i n P\'erez
\and
Antonio Ros}

\date{}

\maketitle

\address{Meeks: Mathematics Department,
University of Massachusetts, Amherst, MA 01003; \email{profmeeks@gmail.com}
\and
Mira: Department of
Applied Mathematics and Statistics, Universidad Polit\'ecnica de
Cartagena, E-30203 Cartagena, Murcia, Spain; \email{pablo.mira@upct.es}
\and
P\'erez: Department of Geometry and Topology,
University of Granada, 18001 Granada, Spain; \email{jperez@ugr.es}
\and
Ros: Department of Geometry and Topology,
University of Granada, 18001 Granada, Spain; \email{aros@ugr.es}}



\begin{abstract}
Given a non-compact, simply connected homogeneous three-manifold $X$
and a sequence $\{\Omega_n\}_n$ of isoperimetric domains in $X$ with
volumes tending to infinity, we prove that  as $n\to \infty$: \ben \item  The radii of the $\Omega_n$
tend to infinity. \item The ratios
$\frac{{\rm Area} (\parc \Omega_n)}{{\rm Vol}(\Omega_n)}$ converge to
the Cheeger constant ${\rm Ch} (X)$, which we also prove to be equal to
$2H(X)$ where $H(X)$ is the critical mean curvature of $X$. \item The values of the constant  mean
curvatures $H_n$ of the boundary surfaces
$\parc \Omega_n$ converge to $\frac{1}{2}{\rm Ch} (X)$. \een Furthermore, when
$\mathrm{Ch}(X)$ is
positive, we prove that for $n$ large, $\parc \Omega_n$ is well-approximated  in
a natural sense by the leaves of a certain
foliation of $X$, where every leaf of the foliation is a surface of constant mean curvature  $H(X)$.
\end{abstract}

\section{Introduction.} \label{intro}

Throughout this paper, {\bf $X$ will denote a non-compact, simply
connected homogeneous three-manifold}. Recall that the {\em isoperimetric
profile} of
 $X$ is defined as the function $I\colon
(0, \infty )\to (0,\infty) $ given by
 \begin{equation}
\label{isop} I(t) = \inf \{ \mbox{Area}(\partial \cD )\ : \
\overline{\cD }\subset X \mbox{ is a smooth compact domain with
Volume}(\cD )=t \}.
\end{equation} The
function $I(t)$ has been extensively studied; see the background
Section~\ref{secprel} for a brief summary of some properties of
$I(t)$.
For every value of $t\in (0,\infty)$, there exists at
least one smooth compact domain $\Omega\subset X$ of volume $t$ and area $I(t)$, and
this domain has boundary
 of non-negative
constant mean curvature with respect to the inward pointing unit normal vector,
due to the fact that $X$ is homogeneous.
Such a smooth compact domain $\Omega$ with smallest boundary area for its given
volume is called an {\em isoperimetric domain} of $X$.

In this paper we will study the geometry of isoperimetric domains in
$X$ of large volume.  The following definitions provide three key geometric invariants
which we will study in detail in order to describe the geometry of these
isoperimetric domains.

\begin{definition}
{\rm Let $\mathcal{A}$ be the collection of all compact, immersed
orientable surfaces in $X$, and given a surface $\Sigma \in
\mathcal{A}$, let $|H_{\Sigma}|\colon \Sigma \to [0,\8) $ stand for the
absolute mean curvature function of $\Sigma $. The {\em critical
mean curvature} of $X$ is defined as
\begin{equation}
\label{Cheeger} H(X)=\inf \{ \max _{\Sigma }|H_{\Sigma}| \ : \
\Sigma \in \mathcal{A} \} .
\end{equation}
}
\end{definition}

\begin{definition}
{\rm The {\it radius} of a compact Riemannian manifold with boundary
is the maximum distance from points in the manifold to its boundary.
}
\end{definition}

One goal of this paper is to prove Theorem~\ref{t1} below which
relates the mean curvatures of boundaries of isoperimetric domains
with large volume to the Cheeger constant of $X$, defined as follows.
\begin{definition}
{\rm The {\em Cheeger constant} of a Riemannian manifold $Y$ with infinite
volume is
\begin{equation}
\label{eq:Chconstant}
 {{\rm Ch}(Y)= \inf \left\{
\frac{\mbox{Area}(\partial \mathcal{D}
)}{\mbox{Vol}(\mathcal{D})}\, :\, \overline{\cD }\subset Y \mbox{ is
a smooth compact domain} \right \} } .
\end{equation}
}
\end{definition}
\vspace{.1cm}

By definition of the Cheeger constant,  $\mathrm{Ch}(X)=\inf\{ \frac{I(t)}{t} \mid t\in
(0,\infty)\}$ for every non-compact, simply connected homogeneous
three-manifold $X$, where $I$ is the isoperimetric profile of $X$.

\begin{theorem}
\label{t1} Let $X$ be a non-compact, simply connected homogeneous three-manifold.
\ben[(1)]
\item
Suppose that
$X$ is not isometric to  the Riemannian product
$\esf^2(\kappa)\times \R$ of a two-sphere of constant curvature
$\kappa>0$ with the real line. If $\Omega \subset X$ is an isoperimetric domain in $X$
with volume $t$, then $\partial \Omega$ is connected
and $$ \mathrm{Ch}(X)<\min\left\{ 2H , \frac{I(t)}{t} \right\},$$
 where
$H>0$ is the constant mean curvature of the boundary of $\Omega $
with respect to the inward pointing unit normal.
\item $\mathrm{Ch}(X)=2H(X)=\lim_{t \to \infty}\frac{I(t)}{t}$.
\item Given any sequence of isoperimetric domains $\Omega _n\subset X$ with
volumes tending to infinity, as $n\to \infty $, the sequence of constant mean curvatures
of their boundaries converges to
$H(X)$ and the sequence of radii of these domains diverges
to infinity.
\een
\end{theorem}

Theorem~\ref{t1} provides additional information about the isoperimetric profile of a
non-compact, simply connected homogeneous three-manifold, see Corollary~\ref{corol1.5}.

The organization of the paper is as follows. In
Section~\ref{secprel} we provide some background information about
the isoperimetric profile of a non-compact, simply connected,
homogeneous three-manifold $X$, and prove some parts of items~(1)
and (2) of Theorem~\ref{t1}.
The rest of the proof of Theorem~\ref{t1}
is divided into two cases; in Section~\ref{secCh0} we consider
the case when the Cheeger constant of $X$ is zero, and in
Sections~\ref{secCh+}, \ref{sec5}, \ref{sec6} we will deal with the case when
$\mathrm{Ch}(X)>0$. When
$\mathrm{Ch}(X)$ is positive, we prove in Theorems~\ref{corcon}
and
\ref{corcon3} that
the boundaries of
isoperimetric domains of large volume in $X$ are in a natural sense
well-approximated by the leaves of a certain foliation $\cF$ of $X$,
where these leaves are surfaces of constant mean curvature $H(X)$.
As a matter of fact, we will prove in Section~\ref{sec6} the following existence result.

\begin{theorem}
\label{thm1.6}
Let $X$ be a homogeneous three-manifold diffeomorphic to $\R^3$.
Then, there exists a foliation $\cF$ of $X$ by surfaces of constant
mean curvature $H(X)$ with the following properties.
\ben[(1)]
\item There exist a 1-parameter subgroup $\G $ and a $(\Z \times \Z)$-subgroup $\Delta $
of the isometry group \mbox{\rm Iso}$(X)$ of $X$, both acting freely on $X$,
such that each of the leaves of $\cF$ is invariant under $\G $ and $\Delta $.
\item All leaves of $\cF$ are congruent in $X$; more precisely, there exists a
1-parameter subgroup $\wt{\G }$ of \mbox{\rm Iso}$(X)$ acting freely on $X$ such that
$\cF =\{ \phi (\Sigma )\ | \ \phi \in \wt{\G }\} $, where $\Sigma $ is any particular leaf of
$\cF $.
\item Each orbit of the action of $\wt{\G }$ on $X$ intersects every leaf of $\cF $ transversely at a
single point.
\een
\end{theorem}

Our interest in results like the ones described above arises from our paper~\cite{mmpr1}, where we
classify the moduli space of constant mean curvature spheres in a
homogeneous three-manifold $X$ in terms of $H(X)$.
%
The
work in~\cite{mmpr1} and in the present paper support the following conjecture:

\begin{conjecture}[Uniqueness of Isoperimetric Domains Conjecture] \label{conj}
If  $X$ is a homogeneous three-manifold
diffeomorphic to $\rth$, then for each $V>0$, there exists a
{\em unique} (up to congruencies) isoperimetric domain
$\Omega(V)$ in $X$ with volume $V$. Furthermore,
$\Omega(V)$ is topologically a compact ball.
\end{conjecture}

%
%

Several results in this paper admit generalizations to the
$n$-dimensional case, where one assumes that $n\leq 7$ in order to avoid lack of
regularity of the boundaries of isoperimetric domains.
Nevertheless, for the sake of simplicity we will develop here only the
case where the ambient homogeneous manifold is three-dimensional.

\section{Background material.}
\label{secprel}

\subsection{Isoperimetric profile.}
Since $X$ is homogeneous and the dimension of $X$ is less than 8,
then by regularity and existence results in geometric measure
theory, for each $t>0$, there exists at least one smooth compact
domain $\Omega \subset X$ that is a solution to the isoperimetric
problem in $X$ with volume $t$, i.e., whose boundary surface
$\partial \Omega $ minimizes area among all boundaries of
domains in $X$ with volume~$t$; as usual, we  call such $\Omega $ an
{\it isoperimetric domain.} It is well-known that the (possibly non-connected)
boundary $\partial \Omega $ of an isoperimetric domain $\Omega $ has constant
mean curvature. In the sequel, we will always orient $\partial \Omega $ with respect to
the inward pointing unit normal vector of $\Omega$.

For any $\ve>0$ and
$t\geq\ve$, there exist uniform estimates for the norm of the second
fundamental forms of the isoperimetric surfaces\footnote{An
{\em isoperimetric surface} is the boundary of an isoperimetric domain.}
enclosing volume $t$, see e.g.,  the proof of Theorem 18 in
the survey paper by Ros~\cite{ros10} for a sketch of proof of this
curvature estimate.

Consider the isoperimetric profile $I\colon
(0, \infty )\to (0,\infty) $ of
 $X$ defined in (\ref{isop}). This profile has been extensively studied
 in greater generality. We next recall some basic properties of it, see
e.g., Bavard and Pansu~\cite{bp1}, Gallot~\cite{gll1} and
Ros~\cite{ros10}:
\begin{enumerate}[(I)]
\item $I$ is locally Lipschitz. In particular, its derivative $I'$
exists almost everywhere in $(0,\infty )$ and for every $0<t_0\leq t_1$,
\[
I(t_1)-I(t_0)=\int _{t_0}^{t_1}I'(t)\, dt.
\]
\item $I$ has left and right derivatives $I'_-(t)$ and $I'_+(t)$ for any
value of $t\in (0,\infty )$. Moreover if $H$ is the mean curvature
of an isoperimetric surface $\partial \Omega $ with Volume$(\Omega
)=t$ (with the notation above), then $I'_+(t)\leq 2H\leq I'_-(t)$.
\item The limit as $t\to 0^+$ of $\displaystyle \frac{I(t)}{(36\pi t^2)^{1/3}}$ is 1.
\end{enumerate}

\begin{remark}
\label{remark1}
{\rm
\ben[(i)]
\item Every  isoperimetric domain $\Omega$ in a non-compact
homogeneous three-manifold is connected; otherwise translate one  component of $\Omega$
 along a continuous 1-parameter family of isometries until it touches
 another component tangentially a first time to obtain a contradiction
to boundary regularity of solutions to isoperimetric domains.

\item From the discussion at the beginning of Section~\ref{subsec2.2},
if $X$ is not isometric to $\S^2(\kappa)\times\R$,
 then $X$ is diffeomorphic to $\rth$. Suppose that $X$ is diffeomorphic to $\rth$ and
 $\Omega$ is an isoperimetric domain  in $X$.
As $\Omega $ is a connected compact domain in $X\approx\rth$, the boundary $\partial \Omega $
contains a unique outer boundary component $\partial$ (here, outer means the component of the
boundary of $\Omega$ that is contained in the boundary of the unbounded component
of $X -\Omega$). We claim that $\partial \Omega =\partial $ and that $\partial \Omega $
has positive mean curvature with respect to the inward pointing normal vector to $\Omega $.
Since connected, compact embedded surfaces in $\rth$  bound unique smooth
compact regions, we can translate a superimposed copy
of the compact region $\Omega (\partial )\subset X$ enclosed by $\partial $
along a 1-parameter group of ambient isometries until the translated copy intersects
$\Omega (\partial )$ a last time at some point $p$. At this last point $p$ of contact,
 the outer boundaries of the two intersecting domains intersect on opposite sides
of their common tangent plane at $p$. The maximum principle
for constant mean curvature surfaces applied at $p$ demonstrates that the mean curvature of $\partial$
is positive with respect to the inward pointing normal of $\Omega$.
Another simple continuous translation argument of a possible interior boundary component $\partial '$
of $\Omega$,  applied in the interior of $\Omega (\partial )$,
gives a contradiction that implies that
the boundary of an isoperimetric domain $\Omega$ is equal to
its outer boundary component; hence $\partial \Omega$ is connected.

\item In the case
that $X$ is isometric to $\esf^2(\kappa)\times \R$, similar
arguments to those appearing in item~(ii) show that the boundary of an isoperimetric
domain $\Omega$ in $X$ is either connected with constant
positive mean curvature, or $\Omega$ is
the product domain $\esf^2(\kappa)\times [R_1,R_2] $ for some $R_1<R_2$. In fact,
Pedrosa~\cite{ped1} proved that there exists $V_0>0$ such that
if $\Omega \subset X$ is an isoperimetric domain with Vol$(\Omega )=V>0$,
then $\Omega $ is a rotationally symmetric ball if $V<V_0$, and $\Omega =
\esf^2(\kappa)\times [R_1,R_2]$ if $V>V_0$ for some $R_1<R_2$.
 In particular, the isoperimetric profile
of $X$ is constant in the interval $[V_0,\infty )$.
\een
}
\end{remark}

The main goal of this section is to prove the next four lemmas.
\begin{lemma}
\label{lem2.1}
  Given $t\in (0,\infty )$, $\frac{1}{2}I'_+(t)$ (resp. $\frac{1}{2}I'_-(t)$)
  equals the infimum
  (resp. supremum) of the mean curvatures
  of isoperimetric surfaces in $X$ enclosing volume $t$. In fact, this
infimum (resp. supremum) is achieved for some isoperimetric domain. Furthermore, the function
$I$ is non-decreasing and strictly increasing when $X$ is diffeomorphic to $\R^3$.
\end{lemma}
\begin{proof}
We will prove the stated properties for the case of $I'_+(t)$ and
leave the similar case of $I'_-(t)$ to the reader. Fix $t_0\in
(0,\infty )$ and let $\mathcal{B}$ be the family of isoperimetric
surfaces in $X$ enclosing volume $t_0$. Since $I$ is locally
Lipschitz, then
$I$ is differentiable in
$[t_0,t_0+1]-A$, where $A\subset [t_0,t_0+1]$ is a set of measure
zero, and the function $t\in [t_0,t_0+1]\mapsto I'(t)$ is integrable
in $[t_0,t_0+1]$.

We claim that there exists a sequence $t_n\in (t_0,t_0+1]-A$
converging to $t_0$ such that $I'(t_n)<I_+'(t_0)+\frac{1}{n}$;
otherwise, there exist numbers $\ve ,\de >0$ such that $I'(t)\geq
I_+'(t_0)+\de $ for all $t\in (t_0,t_0+\ve )-A$. This inequality
implies that for $t\in (t_0,t_0+\ve )$,
\[
\frac{I(t)-I(t_0)}{t-t_0}=\frac{1}{t-t_0}\int _{t_0}^tI'(t)\, dt\geq
I_+'(t_0)+\de ,
\]
which contradicts the existence of the right derivative of $I$
at $t_0$ and therefore, proves our claim.

Given $n\in \N$, there exists an isoperimetric domain $\Omega
_n\subset X$ with volume $t_n$. As $t_n\in (t_0,t_0+1]-A$, then $I$
is differentiable at $t_n$ and property (II) stated just before Remark~\ref{remark1} implies that
$I'(t_n)=2H_n$, where $H_n$ denotes the constant mean curvature of
$\partial \Omega _n$. Since $X$ is homogeneous, standard compactness
results imply that after possibly passing to a subsequence, the
$\Omega _n$ converge to an isoperimetric domain $\Omega _{\infty }
\subset X$ enclosing volume $t_0$, and the sequence of  mean curvatures $H_n$ of their boundaries
converge to the mean curvature $H_{\infty }$ of $\partial \Omega
_{\infty }\in \mathcal{B}$ as $n\to  \infty $. Thus, by the claim proved above,
$I'_+(t_0)+\frac{1}{n}>2H_n$ and taking $n\to \infty $,
$I_+'(t_0)\geq
2H_{\infty }\geq 2\inf \{ H(\Sigma )\ | \ \Sigma \in \mathcal{B}\}
$, where $H(\Sigma )$ denotes the constant mean curvature of $\Sigma
\in \mathcal{B}$. The inequality $I_+'(t_0)\leq 2\inf \{ H(\Sigma )\
| \ \Sigma \in \mathcal{B}\} $ follows directly from property (II)
above, which completes the proof of the first sentence of the lemma.
The second statement (that the infimum is achieved) holds since
$I_+'(t_0)=2 H_{\infty }$, which is the mean curvature of the
isoperimetric domain $\Omega _{\infty }$.

Recall from Remark~\ref{remark1} that the mean curvature of the boundary of
every isoperimetric domain in $X$ is non-negative, and it is positive when $X$
is diffeomorphic to $\R^3$. As for every $t>0$, $I'_+(t)$ is twice the mean curvature of some
isoperimetric domain in $X$ enclosing volume $t$, then $0\leq I'_+(t)\leq I'_-(t)$ for every $t>0$,
from where one deduces that $I$ is non-decreasing. In the case $X$ is diffeomorphic to $\R^3$,
the same argument gives that $I$ is strictly increasing and the proof is complete.
\end{proof}

\begin{lemma}
\label{lem2.2} With the notation above,
\begin{enumerate}[(1)]
\item  Given $t>0$ and $n\in \N$, we have $I(nt)<n\, I(t)$.
\item $\mathrm{Ch}(X)<\frac{I(t)}{t}$ for all $t>0$.
\end{enumerate}
\end{lemma}
\begin{proof}
Given $t>0$, let $\Omega \subset X$ be an isoperimetric domain with
volume $t$. Since $\Omega $
  is compact and $X$ is non-compact and homogeneous, there exists an
  isometry $\phi $ of $X$ such that
$\phi (\Omega )$ is disjoint from $\Omega $.
As $\Omega \cup \phi (\Omega )$ has volume $2t$ but it is not an isoperimetric
domain (since it is not connected), then
\[
I(2t)<\mbox{Area}[\partial (\Omega \cup \phi (\Omega ))] =2\mbox{
Area}(\partial \Omega )=2I(t).
\]
Item~(1) follows by applying this argument to a collection $\{\Omega, \phi_1(\Omega),
\phi_2(\Omega), \ldots ,$ $ \phi_{n-1}(\Omega)\}$ of pairwise
disjoint translated copies of $\Omega$.

By definition of Ch$(X)$, we have Ch$(X)\leq \frac{I(t)}{t}$ for
each $t$. If Ch$(X)$ were equal to $\frac{I(t_0)}{t_0}$ for some
$t_0>0$, then item~(1) of this lemma would imply that
\[
\frac{I(2t_0)}{2t_0}<\frac{2\,
I(t_0)}{2t_0}=\frac{I(t_0)}{t_0}=\mbox{Ch}(X),
\]
which is absurd. This proves $\mathrm{Ch}(X)<\frac{I(t)}{t}$ for all $t>0$.
%
\end{proof}

\begin{lemma}
\label{lem2.3} For any $\ve>0$, there exists $V_{\ve}>0$ such that
for every isoperimetric domain $\Omega \subset X$  with volume
greater than $V_{\ve}$,
\begin{equation}
\label{eq:lem2.3} \mathrm{Ch}(X)<\frac{\mathrm{Area}(\partial
\Omega)} {\mathrm{Vol}(\Omega )}<\mathrm{Ch}(X) +\ve.
\end{equation}
\end{lemma}
\begin{proof}
The first inequality in (\ref{eq:lem2.3}) holds for every
isoperimetric domain by item~(2) of Lemma~\ref{lem2.2}.

Given $\ve >0$, consider a compact domain $\cD _0\subset X$ with volume $V_0$, such that
$\mathrm{Ch}(X)<\frac{\mathrm{Area}(\partial \cD _0)} {V_0}
<\mathrm{Ch}(X)+\frac{\ve }{2}$. Such domain $\cD_0$ exists by definition of
Ch$(X)$ and by item~(2) of Lemma~\ref{lem2.2}. Observe that $I(V_0)\leq \mathrm{Area}(\partial \cD _0)
<V_0(\mathrm{Ch}(X)+\frac{\ve }{2})$. Consider the piecewise constant
function $F\colon (0,\infty )\to \R $ given by
\[
F(t)=(k+1)V_0(\mbox{Ch}(X)+\textstyle{\frac{\ve }{2}}) \mbox{ if }t\in (kV_0,(k+1)V_0], \mbox{ for any $k\in \N$.}
\]
Fix $k\in \N$. By item (1) of Lemma~\ref{lem2.2}, $I((k+1)V_0)<(k+1)I(V_0)<F((k+1)V_0)$.
By Lemma~\ref{lem2.1} $I$ is non-decreasing, and so, $I\leq F$ in $(kV_0,(k+1)V_0]$ for every $k$; thus,
$I\leq F$ in $(0,\infty )$.

On the other hand, a straightforward computation shows that the function $F$ lies below the linear function
$t>0\mapsto (\mbox{Ch}(X)+\ve )t$ for $t\geq 2+\frac{2}{\ve }\mbox{Ch}(X)$. Finally,
$\frac{I(t)}{t}\leq \frac{F(t)}{t}\leq \mbox{Ch}(X)+\ve $ for $t\geq 2+\frac{2}{\ve }\mbox{Ch}(X)$
and the proof is complete.
%
\end{proof}

\begin{lemma}
\label{lem2.4} For each $n\in\N$, there exists $T_n>n$ such that for
every isoperimetric domain $\Omega _n\subset X$ with volume $T_n$,
the mean curvature $H_n$ of its boundary satisfies
$H_n<\frac{1}{2}\mathrm{Ch(X)}+\frac{1}{n}$. In particular,
$2H(X)\leq \mathrm{Ch}(X)$.
\end{lemma}
\begin{proof}
By Lemma~\ref{lem2.3}, given $n\in \N $ there exists $t_n>0$ such
that for every isoperimetric domain $\Omega \subset X$ with volume
greater than $t_n$, we have
\begin{equation}
\label{eq:me2.4a} \mbox{Ch}(X)<\frac{\mbox{Area}(\partial \Omega
)}{\mbox{Vol}(\Omega )} <\mbox{Ch}(X)+\frac{1}{n}.
\end{equation}
Clearly  we can assume $t_n>n$ without loss of generality.

 We claim that for each $n$ there exists $T_n\in (t_n,\infty ) $ such that the
left derivative $I'_-(t)$ of $I$ satisfies
$I'_-(T_n)<\mbox{Ch}(X)+\frac{2}{n}$. Arguing by contradiction,
suppose that $I'_-(t)\geq \mbox{Ch}(X)+\frac{2}{n}$ for all $t\in
(t_n,\infty )$. Then, given $t>t_n$ we have
\begin{equation}
\label{eq:me2.4b} I(t)-I(t_n)=\int _{t_n}^tI'(s)\, ds\geq \left[
\mbox{Ch}(X)+\frac{2}{n}\right] (t-t_n)
\end{equation}
Using (\ref{eq:me2.4a}) and (\ref{eq:me2.4b}), we have
\begin{equation}
\label{eq:me2.4c} \mbox{Ch}(X)+\frac{1}{n}>\frac{I(t)}{t}\geq \left[
\mbox{Ch}(X)+\frac{2}{n}\right] \left( 1-\frac{t_n}{t}\right)
+\frac{I(t_n)}{t}.
\end{equation}
Taking $t\to \infty $ (with $n$ fixed) in (\ref{eq:me2.4c}) and
simplifying we obtain $\frac{1}{n}\geq \frac{2}{n}$, which is
absurd. Hence, our claim holds.

We finish by proving that the statement of the lemma holds for the
value $T_n$ found in the last paragraph. Given an isoperimetric
domain $\Omega _n\subset X$ with volume $T_n$, we can apply property
(II) stated just before Remark~\ref{remark1} to the mean curvature $H_n$ of
the boundary $\partial \Omega _n$ and then apply the claim in the
last paragraph in order to get
\[
2H_n\leq I'_-(T_n)<\mbox{Ch}(X)+\frac{2}{n},
\]
which completes the proof of the lemma.
\end{proof}

\subsection{Classification of the ambient manifolds.}
\label{subsec2.2}
As $X$ is three-dimensional, simply connected and
homogeneous, then $X$ is isometric either to the Riemannian product
$\esf^2(\kappa)\times \R$ of a 2-sphere of constant curvature
$\kappa>0$ with the real line or to a {\it metric Lie group,} i.e., a
Lie group equipped with a left invariant metric (see e.g.,
Theorem~2.4 in~\cite{mpe11}). A special case of this second
possibility is that $X$ is isometric to a semidirect product
$\R^2\rtimes _A\R $ for some $2\times 2$ real matrix $A$ endowed
with its {\it canonical metric;} this means that the group structure is
given on $\R^2\times \R$ by the operation
\begin{equation}
\label{eq:operation}
 ({\bf p}_1,z_1)*({\bf p}_2,z_2)=({\bf p}_1+e^{z_1A}{\bf p}_2,z_1+z_2),
\end{equation}
where $e^{zA}$ is the usual exponentiation of matrices, and the canonical left
invariant metric on $\R^2\rtimes _A\R $ is the one that extends the
usual inner product in $\R^3=T_{e}(\R^2\rtimes _A\R )$ by left
translation with respect to the above group operation, where
$e=(0,0,0)$. Every
three-dimensional, simply connected non-unimodular Lie group lies in
this semidirect product case, as well as the unimodular groups
$\widetilde{\mbox{\rm E}}(2)$ (the universal cover of the group of
orientation-preserving rigid motions of the Euclidean plane),
Sol$_3$ (also known as $E(1,1)$, the group  of
orientation-preserving rigid motions of the Lorentz-Minkowski plane)
and Nil$_3$ (the Heisenberg group of nilpotent $3\times 3$ real
upper triangular matrices with entries 1 in the diagonal).

The classification of three-dimensional, simply connected Lie groups
ensures that except for the cases listed in the above paragraph, the
remaining Lie groups are $\mathrm{SU}(2)$ (the unitary group,
diffeomorphic to the three-sphere) and $\widetilde{\mbox{\rm SL}}(2,\R
)$ (the universal covering of the special linear group). For details,
see~\cite{mpe11}. Recall that we are assuming in this paper
that $X$ is non-compact, so in particular $X$ cannot be
isometric to a metric Lie group
isomorphic to $\su$.

We will use later the following result, which follows rather easily
from the work of Peyerimhoff and Samiou~\cite{pesa1}; see also
Theorem~3.32 in~\cite{mpe11} for a
self-contained proof.
\begin{proposition}
\label{propos2.5} Suppose that $X$ is isometric to a semidirect
product $\R^2\rtimes _A\R$ endowed with its canonical metric.  Then,
$\mathrm{Ch}(X)=\mathrm{ trace}(A)=2H(X)$. Furthermore, $X$ is
unimodular if and only if $\mathrm{Ch}(X)=0$.
\end{proposition}
In particular, the equality ${\rm Ch} (X)=2H(X)$ in item~(2) of Theorem~\ref{t1} holds
when X is isometric to a semidirect product with its canonical
metric.

\section{Isoperimetric domains when $\mathrm{Ch}(X)=0$.}
\label{secCh0}
\begin{theorem}
\label{t2} Let $X$ be a non-compact, simply connected, homogeneous
three-manifold and let $\{ \Omega_n\} _n$ be any sequence of isoperimetric
domains in $X$ with volumes tending to infinity. Then, the following
statements are equivalent:
\begin{enumerate}[(A)]
\item $\mathrm{Ch}(X)=0$.
\item $X$ is isometric either to the Riemannian product $\esf^2(\kappa)\times \R$ of a
2-sphere of constant curvature $\kappa>0$ with the real line, or to
a semidirect product $\R^2\rtimes _A\R$ for some $2\times 2$ real
matrix $A$ with trace zero,
endowed with its canonical metric. \vspace{.1cm}
\item  ${\displaystyle \lim_{n\to \infty}
\frac{\mathrm{Area}(\partial \Omega _n)}{\mathrm{Vol}(\Omega
_n)}=0}$.
\item The mean  curvatures $H_n$
of $\partial \Omega_n$  are non-negative and satisfy $\lim_{n\to \infty} H_n=0$.
\vspace{.1cm}
\item
\label{E} Given any $R>0$, $\ds  \lim_{n\to \infty}
\frac{\mathrm{Vol}(\Omega _n (R))}{\mathrm{Vol}(\Omega
_n)}=0$, where $\Omega _n(R)=\{x\in \Omega_n \mid \mbox{\rm dist}_X(x,\partial
\Omega_n)<R\}$. \een

\vspace{.25cm} \noindent Furthermore, if any of the above conditions
hold, then the sequence of radii of the $\Omega _n$ tends to
infinity as $n\to \infty $.
\end{theorem}

\begin{proof}[Proof of Theorem~\ref{t2}]
We first observe that if item~(E) in Theorem~\ref{t2} holds, then
the last property in the statement of the theorem
concerning the radii of the domains $\Omega_n$ also holds:
otherwise, after passing to a subsequence there exists $C>0$ such that for all $x\in \Omega _n$ and
for all $n\in \N$, $\mbox{\rm dist}_X(x,\parc\Omega _n)\leq C$. This implies that $\Omega
_n(R)=\Omega _n$ for all $R>C$, which contradicts the assumption in
item~(E).

We will prove the equivalence between the items in Theorem~\ref{t2}
in six steps.
 \par
\vspace{.2cm} {\bf Step 1:} {\it If $X$ is isometric to the Riemannian
product $\esf^2(\kappa )\times \R $ for some $\kappa
>0$, then items~(A), (B), (C), (D) and (E) in
Theorem~\ref{t2} hold. }
 \begin{proof}[Proof of Step 1.]
Consider the smooth domains $\Omega _R=\esf^2(\kappa)\times [0,R]$
with $R>0$. Then,
\begin{equation}
\label{eq:t2a}
 \frac{\mbox{Area}(\partial \Omega _R)}{\mbox{Vol}(\Omega _R)}=
\frac{2\mbox{ Area}(\esf^2(\kappa ))}{R\mbox{ Area}(\esf^2(\kappa
))} \to 0\quad \mbox{ as }R\to \infty , \end{equation} and thus,
Ch$(X)=0$ and item~(A) of Theorem~\ref{t2} holds. By item (iii) of Remark~\ref{remark1},
%
%
(\ref{eq:t2a}) gives that items~(C) and (E) of
Theorem~\ref{t2} hold. As the boundary of $\Omega _R$ is minimal for $R$ sufficiently large,
then item~(D) of the theorem also holds.  Item (B) holds by the assumption that $X$ is isometric to
$\esf^2(\kappa )\times \R $.
 \end{proof}
 \par
\vspace{.2cm} {\bf Step 2:} {\it Item~(A) and item~(B) in Theorem~\ref{t2}
are equivalent.}
 \begin{proof}[Proof of Step 2.]
A result by  Hoke~\cite{hoke1} states that Ch$(G) =
0$ for a non-compact, simply connected Lie group $G$ (of any
dimension) with a left invariant metric if and only if $G$ is
unimodular and amenable. Therefore, Step 2 follows
from  Step 1, Proposition~\ref{propos2.5} and this result
by Hoke, since $\sl $ is not amenable.
\end{proof}
 \par
\vspace{.2cm} {\bf Step 3:} {\it Items (A) and (C) in Theorem~\ref{t2}
are equivalent.}
 \begin{proof}[Proof of Step 3.]
This follows directly from
Lemma~\ref{lem2.3}.
\end{proof}
 \par
\vspace{.2cm} {\bf Step 4:} {\it Item~(D) in Theorem~\ref{t2}
implies item~(A). }
 \begin{proof}[Proof of Step 4.]
 Assume that item~(D) in Theorem~\ref{t2} holds.
Arguing by contradiction, suppose that Ch$(X)>0$. Choose $V_0>0$
sufficiently large so that for every isoperimetric domain $\Omega
\subset X$ with volume at least $V_0$, the mean curvature of
$\partial \Omega $ is at most Ch$(X)/4$ (this $V_0$ exists by our
hypothesis in item~(D)). Hence given $a>0$, {\small
\[
\frac{I(V_0+a)-I(V_0)}{V_0+a}=\frac{1}{V_0+a}\int
_{V_0}^{V_0+a}I'_+(t)\, dt
\]
\[
\stackrel{(\star )}{\leq }\frac{1}{V_0+a}\int
_{V_0}^{V_0+a}2 \frac{\mbox{Ch}(X)}{4}\, dt =\frac{a\,
\mbox{Ch}(X)}{2( V_0+a)} < \frac{\mbox{Ch}(X)}{2},
\] }
where we have used in $(\star )$ our hypothesis on the mean
curvature of isoperimetric domains with volume at least $V_0$ and
property~(II) stated just before Remark~\ref{remark1}.
Since $\displaystyle \frac{I(V_0+a)}{V_0+a}>\mbox{Ch}(X)$ by  item~(2)
of Lemma~\ref{lem2.2}, the  inequalities
 \[
\mbox{Ch}(X)-\frac{I(V_0)}{V_0+a} <\frac{I(V_0+a)-I(V_0)}{V_0+a}
< \frac{\mbox{Ch}(X)}{2}.
\]
cannot hold for $a$ sufficiently large under our assumption that
$\mbox{Ch}(X)>0 $.
 This contradiction proves  $\mbox{Ch}(X)=0 $, which completes the proof of this step.
\end{proof}

\vspace{.2cm} {\bf Step 5:} {\it Item~(E) and item~(A) in
Theorem~\ref{t2} are equivalent. }
 \begin{proof}[Proof of Step 5.]
We first check that item~(E) implies item~(A), so assume that item~(E) holds.

As the volumes of the
isoperimetric domains $\Omega _n$ are larger  than some fixed $V>0$, then
their boundary surfaces $\partial \Omega _n$ have uniformly bounded
second
 fundamental form (and hence also uniformly bounded constant mean
 curvatures $H_n$). We consider two possible cases.

If item~(D) also holds, then by Step 4 we see that item~(A) holds.
Now assume that item~(D) fails to hold in general. Thus,
there is a sequence $\{\Omega'_n\}_n$ of isoperimetric domains with
volumes tending to infinity, such that the mean curvatures
$H_n'$ of $\parc \Omega'_n$ are bounded away from zero.
 Since the sectional curvature of $X$ is bounded from above,
Theorem~3.5 in Meeks and Tinaglia~\cite{mt3} then ensures the
existence of a positive number $\de $ such that $\partial \Omega' _n$
has a regular $\de$-neighborhood in $\Omega' _n$, i.e., the geodesic
segments in $\Omega' _n$ of length $\de$ and normal to $\partial
\Omega'_n$ at one of their end points are embedded segments and they do not intersect each other, for each
$n\in \N$. Since the surfaces $\partial \Omega'_n$ have uniformly bounded second
fundamental forms and the absolute sectional curvature of $X$ is bounded,
there is a constant $c>0$
such that $\mathrm{Vol}(\Omega'_n(\de))\geq c
\,\mathrm{Area}(\partial \Omega'_n)$, for all $n\in \N$. Thus,
\begin{equation}
\label{eq:step4a} \mbox{Ch}(X)\leq \frac{\mbox{Area}(\partial \Omega'
_n)}{\mbox{Vol}(\Omega' _n)}\leq
\frac{1}{c}\frac{\mbox{Vol}(\Omega' _n(\de))}{\mbox{Vol}(\Omega'
_n)}, \quad \mbox{ for all $n\in \N$.}
\end{equation}
As we are assuming that item~(E) in Theorem~\ref{t2} holds, then the
right-hand-side of~(\ref{eq:step4a}) tends to zero as $n\to \infty
$. Thus, Ch$(X)=0$ and we conclude that item~(E) implies item~(A).

We next check that item~(A) implies item~(E). Fix $R>0$ and note
that given $n\in \N$ and $p\in \Omega _n(R)$, the closed metric ball in $X$
\[
\overline{\B }(p,R)=\{ x\in X \ | \ \mbox{\rm dist}_X(x,p)\leq R\}
\]
intersects the boundary $\partial \Omega _n$, where $d$ denotes the
distance function in $X$ associated to its Riemannian metric. For
$n\in \N$ fixed, consider the set $\mathcal{A}_n$ whose elements are
the pairwise disjoint collections $\{ \overline{\B }(p_i,2R)\ | \ i=1,\ldots
,k\} $ for some collection of points $p_1,\ldots ,p_k\in \Omega _n(R)$ for some $k\in \N$. Note
that $\mathcal{A}_n$ is non-empty and it can be endowed with the
partial order given by inclusion. Since $\ov{\Omega _n(R)}$ is
compact, there exists a maximal element in $\mathcal{A}_n$ for this
partial order. In other words, for any fixed $n\in \N$,  there exists a finite set $\{
p_1,\ldots ,p_{k(n)}\} \subset \Omega _n(R)$ such that
$\mathcal{C}_n= \{ \overline{\B}(p_1,2R),\ldots ,\overline{\B }(p_{k(n)},2R)\} $ is a
maximal collection of pairwise disjoint closed balls with centers in
$\Omega _n(R)$ and fixed radius $2R$.
 This maximality implies that if $x\in \Omega _n(R)$, then there
 exists some $i\in\{1,\ldots ,k(n)\}$ such that
 $\overline{\B }(x,2R)\cap \overline{\B }(p_i,2R)\neq \mbox{\O }$.
Consequently, the triangle inequality gives that the collection
\[
\mathcal{C}_n'=\{ \overline{\B}(p_1,4R),\ldots ,\overline{\B }(p_{k(n)},4R)\}
\]
is a covering of $\Omega _n(R)$, and thus,
\begin{eqnarray}
\label{eq:aa}
\hspace{1cm}   \mbox{Vol}(\Omega _n(R)) & \leq &
{\displaystyle \mathrm{Vol}\left( \bigcup _{i=1}^{k(n)}\overline{\B }(p_i,4R)\right) }
\\
& \leq &
{\displaystyle \sum_{i=1}^{k(n)} \mathrm{Vol}\left( \overline{\B }(p_i,4R)\right) =
  k(n)\mathrm{ Volume}\left( \overline{\B }(p_1,4R)\right) .}
\end{eqnarray}

As the balls in $\mathcal{C}_n$ are pairwise disjoint, we have
\begin{equation}
\label{eq:ab} \mathrm{Area}(\partial \Omega _n) \geq
 \sum _{i=1}^{k(n)}\mathrm{Area}\left[ \overline{\B }(p_i,2R)\cap \partial \Omega _n\right] .
\end{equation}
Since the norm of the second fundamental form of
$\partial \Omega _n$ is uniformly bounded (independently of $n$),
then there exists some $\tau\in(0,R)$ such that
for all $p\in \parc \Omega_n$, the exponential map on the disk of
radius $\tau$ in $T_p \parc\Omega_n$ is a quasi-isometry\footnote{Recall
that a {\it quasi-isometry} $f\colon (X,g)\to (Y,g')$ between Riemannian
manifolds is a diffeomorphism satisfying $C\, g\leq f^*g'\leq \frac{1}{C}g$ in $X$
for some $C>0$.},
with constant depending only on $X$ and the bound of the second
fundamental form (and not depending on $n$).
Then, since for each $i=1,
\ldots ,k(n)$, $\overline{\B }(p_i,R)$ intersects $\partial \Omega _n$ at some point
$q_i$, then the intrinsic $\tau $-disk centered at $q_i$ has area not less that
some number $\mu >0$ not depending on $n$ or $q_i$. Since the intrinsic distance dominates
the extrinsic distance, then the triangle inequality implies that
\begin{equation}
\label{eq:ac} \mathrm{Area}\left[ \overline{\B }(p_i,2R)\cap \partial \Omega
 _n\right] \geq \mu , \quad \mbox{for all $i=1,\ldots ,k(n)$ and $n\in \N$.}
\end{equation}

Now, (\ref{eq:aa}), (\ref{eq:ab}) and (\ref{eq:ac}) give for all $n\in \N$, the following inequalities:
\begin{equation}
\label{eq:3.5} \frac{\mbox{Area}(\partial \Omega
_n)}{\mbox{Vol}(\Omega _n(R))}\geq \frac{
 \sum _{i=1}^{k(n)}\mathrm{Area}\left[ \overline{\B }(p_i,2R)\cap \partial \Omega _n\right]}{k(n)\mathrm{
Volume}\left( \overline{\B }(p_1,4R)\right)} \geq \frac{\mu } {\mathrm{
Volume}\left( \overline{\B }(p_1,4R)\right) }.
\end{equation}

Assume now that item~(A) holds, and so item~(C) also holds by Step 3, i.e.,
 \[
 \lim_{n\to \infty}\frac{\mbox{Area}(\partial
\Omega _n)}{\mbox{Vol}(\Omega _n)}= 0.
\] So the only way  (\ref{eq:3.5}) can hold is that $\ds  \lim_{n\to \infty}$ {\small $\ds
\frac{\mathrm{Vol}(\Omega _n (R))}{\mathrm{Vol}(\Omega
_n)}=0$}, which finishes the proof of Step 5.
\end{proof}
 \par
\vspace{.2cm} {\bf Step 6:} {\it Item~(C) in Theorem~\ref{t2}
implies item~(D).}
 \begin{proof}[Proof of Step 6.]
Assume that item~(C) holds. Arguing by contradiction, suppose that item~(D) of Theorem~\ref{t2} fails to
hold.  Then Step~1
implies that $X$ is not isometric to $\esf^2(\kappa )\times \R $,
and so, by Remark~\ref{remark1}, $X$ is diffeomorphic to $\R^3$ and the boundaries of
isoperimetric domains are connected with positive mean curvature. Therefore, the failure of item~(D) to hold
implies that
there exists a sequence of isoperimetric domains $\Omega_n$
with volumes tending to infinity for which
the mean curvatures
 $H_n$ of $\partial \Omega_n$ satisfy
$H_n\geq \be $ for some number $\be >0$. As item~(C) holds, then
item~(A) also
holds by Step 3 and so Ch$(X)=0$. By Step 5 we see that item~(E)
also holds, and thus the radii of the $\Omega _n$ diverge to infinity as
$n\to \infty $ by the argument just before the statement of Step~1.

Since Ch$(X)=0,$  Lemma~\ref{lem2.4} implies that there exists a
sequence $\{ T_k\} _k\subset (0,\infty )$ diverging to infinity such
that given a sequence $\{ \Omega _k^1\} _k$ of isoperimetric domains
in $X$ with Volume$(\Omega _k^1)= T_k$, then the mean curvature
$H^1_k$ of the boundary of $\Omega _k^1$ satisfies $H_k^1<1/k$, for
all $k\in \N$. Since each $\Omega _k^1$ is compact and the radii of
the $\Omega _n$ diverge to infinity as $n\to \infty $, then given
$k\in \N$, there exists $n(k)\in \N$ such that, after an ambient
isometry of $X$ applied to $\Omega_{n(k)}$, we have $\Omega _k^1\subset \Omega
_{n(k)}$. As $X$ is homogeneous, a simple application of the mean curvature
comparison principle implies that $H_{n(k)}\leq H_k^1 $ (move isometrically
$\Omega _k^1$ inside $\Omega _{n(k)}$ until the first time that their
boundaries touch). This is a contradiction, since $\be \leq
H_{n(k)}\leq H_k^1<1/k$ for all $k$. This contradiction finishes the
proof of Step~6.
\end{proof}
 \par

Finally, note that Steps 1-6 above complete the proof of
Theorem~\ref{t2}.
\end{proof}

\section{Simply connected homogeneous three-manifolds with $\mathrm{Ch}(X)>0$.}
\label{secCh+}

By the results in Section~\ref{subsec2.2} and by Theorem~\ref{t2},
the condition Ch$(X)>0$ is equivalent to the fact that $X$ is
isometric to a metric Lie group which is either a
non-unimodular semidirect product $\R^2\rtimes _A\R $ endowed with its canonical metric,
or
$\sl $ with a left invariant metric. Thus, {\it from this point on we will assume that $X$ is
identified with the related metric Lie group.} In particular given
$a\in X$, the left translation $l_a\colon  X\to X$ defined by
$l_a(x)=a\, x$ (we will omit the group operation in $X$) is an
isometry. Note that given a right invariant vector field $K$ on $X$,
its associated  1-parameter group of diffeomorphisms is the
1-parameter group of isometries $\{ l_a\ | \ a\in \G \} $, where $\G
\subset X$ is the 1-parameter subgroup of $X$ defined by $\G(0)=e$,
$\G '(0)=K(e)$ (in the sequel, $e$ will denote the identity element of $X$). Therefore,
$K$ is a Killing vector field.

We next recall some properties of such a metric Lie group with Ch$(X)>0$, that will
be useful in later discussions. For detailed proofs of the
properties stated below, see~\cite{mpe11}.
\par
\vspace{.2cm} {\bf Case (A): $X$ is a non-unimodular semidirect
product.}
\newline
Assume that $X=\R^2\rtimes _A\R $ where
\[
A=\left(
\begin{array}{cc}
a & b \\
c & d\end{array}\right) ,
\]
equipped with its canonical left invariant metric.
An orthonormal left invariant frame for
the canonical left invariant metric on $X$ is
\begin{equation}
\label{eq:6}
 E_1(x,y,z)=a_{11}(z)\partial _x+a_{21}(z)\partial _y,\quad
E_2(x,y,z)=a_{12}(z)\partial _x+a_{22}(z)\partial _y,\quad
 E_3=\partial _z,
\end{equation}
where $e^{zA}=\left( a_{ij}(z) \right) _{i,j}=1,2$ and $\partial _x=\frac{\partial }{\partial x},
\partial _y,\partial _z$ is the usual parallelization of $\R^3$.
The Lie bracket is given by
\[
[E_1,E_2]=0,\quad [E_3,E_1]=aE_1+cE_2,\quad [E_3,E_2]=bE_1+dE_2.
\]

In the natural coordinates $(x,y,z)\in \R^2\rtimes _A\R $, the
canonical left invariant metric is given by
\begin{equation}
\label{eq:13}
 \left.
\begin{array}{rcl}
\langle ,\rangle & =&
 \left[ a_{11}(-z)^2+a_{21}(-z)^2\right] dx^2+
\left[ a_{12}(-z)^2+a_{22}(-z)^2\right] dy^2 +dz^2 \\
& + & \rule{0cm}{.5cm} \left[
a_{11}(-z)a_{12}(-z)+a_{21}(-z)a_{22}(-z)\right] \left( dx\otimes
dy+dy\otimes dx\right) ,
\end{array}
\right.
\end{equation}
and the Levi-Civita connection $\nabla $ for the canonical left
invariant metric is given by
\begin{equation}
\label{eq:12}
\begin{array}{l|l|l}
\rule{0cm}{.5cm}
\nabla _{E_1}E_1=a\, E_3 & \nabla _{E_1}E_2=\frac{b+c}{2}\, E_3 & \nabla _{E_1}E_3=-a\, E_1-\frac{b+c}{2}\, E_2 \\
\rule{0cm}{.5cm}
\nabla _{E_2}E_1=\frac{b+c}{2}\, E_3 & \nabla _{E_2}E_2=d\, E_3 & \nabla _{E_2}E_3=-\frac{b+c}{2}\, E_1-d\, E_2 \\
\rule{0cm}{.5cm} \nabla _{E_3}E_1=\frac{c-b}{2}\, E_2 & \nabla
_{E_3}E_2=\frac{b-c}{2}\, E_1 & \nabla _{E_3}E_3=0.
\end{array}
\end{equation}

In particular (\ref{eq:12}) implies that the mean curvature of each
leaf of the foliation $\mathcal{F}= \{ \R^2\rtimes _A\{ z\} \mid
z\in \R \}$ with respect to the unit normal vector field $E_3$ is
the constant $H=\mbox{trace}(A)/2$, which equals $H(X)=\frac12\mathrm{Ch}(X)$ by
Proposition~\ref{propos2.5}.

\par
\vspace{.2cm} {\bf Case (B): $X$ is $\sl $ equipped with a left
invariant metric.}
\newline
The $2\times 2$ real matrices with determinant equal to 1 form the
 {\it special linear group} SL$(2,\R )$, and represent the orientation-preserving
linear transformations of $\R^2$ that preserve the oriented area.
The quotient of SL$(2,\R )$ modulo $\{ \pm \mbox{identity}\} $ is
the {\it projective special linear group} PSL$(2,\R )$, isomorphic
to the group of orientation-preserving isometries of the hyperbolic
plane $\Hip ^2$, and naturally diffeomorphic to the unit tangent bundle of
$\Hip ^2$. The fundamental groups of SL$(2,\R )$, PSL$(2,\R )$ are
infinite cyclic, and the universal cover of both groups is the
simply connected unimodular Lie group $\sl $, which is the
underlying algebraic structure of the homogeneous manifold $X$ in
this case~(B).

PSL$(2,\R )$ admits three types of 1-parameter subgroups, namely
elliptic, parabolic and hyperbolic subgroups, that correspond to
1-parameter subgroups of M\"{o}bius transformations of the
Poincar\'e disk with zero, one or two fixed points at the boundary
at infinity $\partial _{\infty }\Hip ^2= \esf^1$, respectively. The Lie algebra of
$\sl $ (and of SL$(2,\R )$, PSL$(2,\R )$) is the linear space
$\mathfrak{sl}(2,\R )$ of $2\times 2$ real matrices with trace zero,
with the Lie bracket given by the commutator of matrices. The basis
of $\mathfrak{sl}(2,\R )$ described by the matrices
\begin{equation}
  \label{eq:[]PSL2R}
E_1=\left(
\begin{array}{rr}
1 & 0 \\
0 & -1\end{array}\right) ,\quad E_2=\left( \begin{array}{cc}
0 & 1 \\
1 & 0\end{array}\right) ,\quad E_3=\left(
\begin{array}{rr}
0 & -1 \\
1 & 0\end{array}\right)
\end{equation}
satisfies the Lie brackets relations
\[
[E_1,E_2]=-2E_3,\quad [E_2,E_3]=2E_1,\quad [E_3,E_1]=2E_2.
\]

The two-dimensional subgroups of PSL$(2,\R )$ are non-commutative
and form an $\esf^1$-family $\{ \Hip ^2_{{\mu} }\ | \ {\mu} \in \esf^1=\partial
_{\infty }\Hip ^2\} $, where
\begin{equation}
\label{eq:Htheta}
 \Hip ^2_{{\mu} }=\{ \mbox{orientation-preserving isometries of $\Hip ^2$
 which fix }{\mu \in \partial_{\infty }\Hip ^2} \} .
\end{equation}
Elements in $\Hip^2_{{\mu} }$ are rotations around ${\mu} $ (parabolic)
and translations along geodesics one of whose end points is ${\mu} $
(hyperbolic). The one-dimensional and two-dimensional subgroups of $\sl $ are
the lifts via the covering map $\sl \to \mbox{PSL}(2,\R )$ of the
corresponding one-dimensional and two-dimensional subgroups of PSL$(2,\R )$,
and we will use accordingly the notation {\it elliptic, parabolic,
hyperbolic} and $\Hip ^2_{{\mu} }$ for these connected lifted subgroups of $\sl $.
Under its left action, every 1-parameter subgroup of $\sl $ generates a right invariant
Killing vector field on $\sl $, and we will also call  these right
invariant vector fields {\it elliptic, parabolic} and {\it
hyperbolic}, accordingly to the nature of the related 1-parameter
subgroups.

As mentioned earlier, there is a natural projection of $\sl$ to the hyperbolic
plane. More specifically, there is a submersion
\begin{equation}
\label{eq:Pi} \Pi\colon \sl \to \H^2,
\end{equation}
which is the composition of the covering map $\sl \to
\mbox{PSL}(2,\R )$ with the natural projection from $\psl$ to $\H^2$
obtained after identifying $\psl $ with the unit tangent bundle of
$\Hip ^2$.

With respect to the choice of basis~(\ref{eq:[]PSL2R}) for the Lie
algebra $\mathfrak{sl}(2,\R )$, the center of $\sl $ is an infinite
cyclic subgroup contained in the integral curve $\G^E\subset \sl $
of the left invariant vector field $E_3$ in ~(\ref{eq:[]PSL2R}) that passes
through the identity element $e$ of $\sl$; the
image set of $\G^E$ is a 1-parameter elliptic subgroup of $\sl$.

We next fix some notation that we will use in the remainder of this manuscript.
 Let $\G^H$, $\G^E$, $\G^P$ be
the 1-parameter subgroups of $\sl $ given by
\begin{equation}
\label{eq:GHEP}
(\G ^H)'(0)=E_2(e), \qquad (\G^E)'(0)=E_3(e), \qquad (\G^P)'(0)=E_1(e)+E_3(e),
\end{equation}
 where 
$E_1,E_2,E_3$ are given by~(\ref{eq:[]PSL2R}). Thus, $\G^H$ (resp.
$\G^E$, $\G^P$) is a hyperbolic (resp. elliptic, parabolic) 1-parameter subgroup of
$\sl $.
Let
$\t \in \partial _{\infty }\Hip ^2=\esf ^1$ be the end point of the parameterized
geodesic $\Pi (\G^H )$ of $\Hip ^2$ obtained by projecting $\G^H$
via the map $\Pi $ given in~(\ref{eq:Pi}), which is the end point of
$\Pi (\G^H ([0,\infty )))$ at infinity in $\HH^2$. Let $\H^2_{\t }$ be the
lift to $\sl $ of the two-dimensional subgroup of PSL$(2,\R )$ that
consists of the isometries of $\H^2$ that fix~$\t $. $\H ^2_{\t }$ contains
both 1-parameter subgroups $\G^P$ and $\G ^H$. Furthermore, $\G^P$ is
the unique parabolic subgroup and the unique normal 1-parameter subgroup of $\H ^2_{\t }$.
Left translations by elements in $\G^P$ (resp. in $\G^H$)
generate a right invariant parabolic vector field $K^P$
(resp. hyperbolic vector field $K^H$) on $\sl $.

\begin{definition}
\label{defhorocyl}
{\em
 Given a horocycle $\a\subset \HH^2$, we
call $\Pi ^{-1}(\a)\subset \sl$ the {\em horocylinder} in $X$ over $\a$.}
\end{definition}
\begin{figure}
\parbox[t]{.4\linewidth}{\includegraphics[height=7.2cm]{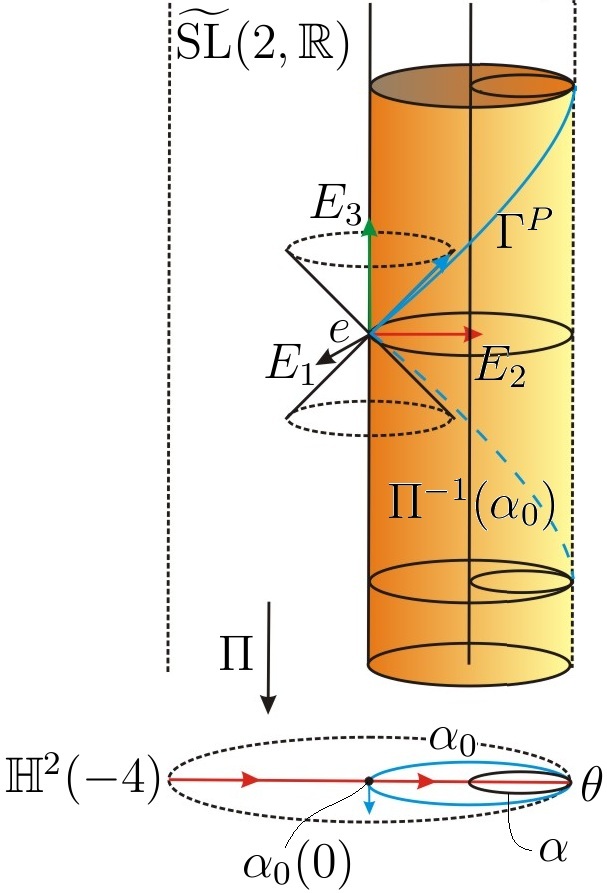}}%
\hspace{.1\linewidth}%
\parbox[b]{.5\linewidth}{%
The shaded surface in $\sl $ is a horocylinder, inverse image by the
projection $\Pi $ of a punctured circle $\a _0\subset \Hip ^2$
tangent at a point in the boundary at infinity of $\Hip ^2$. $\Pi
^{-1}(\a _0)$ is everywhere tangent to the parabolic right invariant
vector field $K^P$ on $\sl $ generated by the 1-parameter parabolic subgroup
$\G^P$.\vspace{1cm}} \caption{} \label{horocyl1}
\end{figure}

Let $\mathcal{C}$   be the set of horocycles in $\HH^2$
tangent to the $\t\in \partial_{\infty } \HH^2$ defined in the previous paragraph. For
each $\a\in \cC$, the horocylinder $\Pi ^{-1}(\a )\subset \sl$ is everywhere tangent to  the parabolic right
invariant vector field $K^P$ of $X$ and  $\Pi
_*[(K^P)(e)]=\a _0'(0)$, where $\a _0$ is the horocycle in
$\mathcal{C}$ that passes through $\Pi (e)\in \Hip^2$, parameterized
appropriately and so that $\a _0(0)=\Pi (e)$, see Figure~\ref{horocyl1}.

The family of left invariant metrics on $\sl$ is three-parametric.
Any such  left invariant metric can be constructed
by declaring that the left invariant vector fields $E_1,E_2,E_3$
in~(\ref{eq:[]PSL2R}) are orthogonal with corresponding lengths
being
 arbitrary positive numbers
$\l _1,\l _2$, $\l _3>0$, respectively.  Thus, we have:

\begin{proposition}
\label{metrics}
The space of left invariant metrics on $\sl$ can be naturally parameterized
by the  open set  $\cM=\{(\l_1,\l_2,\l_3)\in \rth \mid \l_i>0, i=1,2,3\}$,
whereby  we declare the ordered set $\{\l_1E_1,\l_2 E_2, \l_3 E_3 \}$ of
left invariant vector fields on $\sl$
to be orthonormal.
Henceforth, we will identify the space of left invariant metrics on $\sl$ with the set
$\cM$ under the above correspondence.
\end{proposition}

Among the left invariant metrics of $\sl$, we have a two-parameter
subfamily contained in $\cM$ of those metrics that have an isometry group of dimension four;
these special metrics $(\l_1,\l_2,\l_3)\in \cM$  correspond to the case where  $\l_1=\l_2$.
The generic case of a
left invariant metric on $\sl $ has a three-dimensional group of
isometries. For any left invariant metric on $\sl $, the
directions of $E_1,E_2,E_3$ can be proven to be principal directions
for the Ricci tensor, and there exist orientation-preserving
diffeomorphisms of order two around any of the integral curves of these
three vector fields, which turn out to be isometries for every left
invariant metric on $\sl $. Hence, the 1-parameter subgroups
$\G^H$ and $\G^E$ are geodesics in every metric described in
Proposition~\ref{metrics}, as they are the fixed point sets of rotational
isometries in every such metric.

In the case that $X$ is $\sl $ equipped with a left
invariant metric whose isometry group is four-dimensional, then
after rescaling this left invariant metric we can consider $\Pi
\colon X\to \Hip ^2$ to be a Riemannian submersion with constant
bundle curvature.
 In this case, the following property is well-known.
 \begin{lemma}
 \label{lemma4.3}
If $X$ is $\sl $ equipped with a left invariant metric whose isometry group is four-dimensional,
then every horocylinder in $X$ has constant mean curvature equal
to the critical mean curvature\footnote{This follows, for instance,
from the fact that parallel horocylinders produce a foliation
and are limits of spheres with constant mean curvature in
such an $X$. } $H(X)$ of $X$.
 \end{lemma}

In contrast to the statement of Lemma~\ref{lemma4.3}, if $X$ is $\sl $ equipped with a
left invariant metric whose isometry group
is three-dimensional, then for any horocycle $\a$ in $\HH^2$, the horocylinder $\Pi ^{-1}(\a )\subset X$
does not have constant mean
curvature.

\section{Foliations by leaves of critical mean curvature in $\sl $.}
\label{sec5}
In this section we will use the notation developed in the previous section. We will
{\bf assume that $X$ is $\sl$ endowed with an arbitrary left invariant metric} $g$,
and we will construct a topological product foliation of $X$ by
surfaces with constant mean curvature equal to $H(X)$. The existence
and properties of this foliation will be important in the proof of
Theorem~\ref{t1} which we give in Section~\ref{sec6}.


\begin{definition}
\label{def:basis} {\rm Let $a_1$ be a fixed element of
$\G^P-\{e\}$ and let $a_2$ be one of the two generators of the center of $\sl$ (in
particular, $a_2\in \G^E-\{ e\} $). Thus, the left translations
$l_{a_1},l_{a_2}$ generate a discrete subgroup $\Delta $ of the
isometry group of $X=(\sl ,g)$ isomorphic to $\Z \times \Z $ that acts
properly and discontinuously on $\sl $. We define  $W=\sl /\Delta$ and let
\begin{equation}
\label{eq:piW}
\pi_W: \sl \to W
\end{equation}
denote the corresponding quotient submersion.
Note that the definition of
$W$ is independent of the left invariant metric $g$ considered on $\sl$; however,
as $\Delta $ acts on $\sl $ by isometries of $g$, then $W$ can be endowed
with the quotient metric $g_W$ so that
\begin{equation}
\label{eq:gW}
\pi _W:X\to (W,g_W)
\end{equation}
becomes a local isometry and $(W,g_W)$ is locally homogeneous.
}
\end{definition}

The next two lemmas collect some properties that do not depend
on the left invariant metric $g$ on $\sl $. Item~(2) of Lemma~\ref{claim4.3}
implies that $W$ is diffeomorphic to the product of a torus
with the real line, provided that the surface $\Sigma _1\subset \sl$
in the statement of Lemma~\ref{claim4.3} exists. We will prove the
existence of such surface $\Sigma _1$ in Lemma~\ref{lem5.3} below.
\begin{lemma}
\label{claim4.3}
In the above situation, $K^P$ is invariant under the left action of $\Delta$,
and so the left action of $\G^P$ on $X$ induces a left action of $\G^P$ on
 $W$. Let $\Sigma _1\subset \sl $ be a properly embedded surface invariant
  under the left action of $\G^P$ and under $l_{a_2}$. Then:
\ben
\item For all $h\in \G^H$, the group $\Delta $ acts properly and
discontinuously on $l_h(\Sigma _1)$ and the left action of $\G^P$ leaves invariant $l_h(\Sigma _1)$.

\item Suppose that each integral curve of $K^H$ intersects $\Sigma_1$ transversely in a single point. Then
\[
\mathcal{F}(\Sigma _1,\G^H)=
\{ l_h(\Sigma _1) \ | \ h\in \G^H\}
\]
 is a product foliation of $\sl $ and each of the leaves
$l_h(\Sigma _1)$ of $\mathcal{F}(\Sigma _1,\G^H)$ is a properly
embedded topological plane invariant under the left action
of $\G^P$ and under $l_{a_2}$. In particular, $\mathcal{F}(\Sigma _1,\G^H)$
descends to the product foliation $\{ l_h(\Sigma _1)/\Delta \ | \ h\in \G^H\} $ of $W$,
 each of whose leaves
$l_h(\Sigma _1)/\Delta$ is a torus invariant
under the induced action of $\G^P$ on $W$.
\een

\end{lemma}
\begin{proof}
$K^P$ is $l_{a_1}$-invariant, since $a_1\in \G^P$ and the left action of $\G^P$ generates $K^P$.
As $a_2$ lies in the center of $\sl $, then $l_{a_2}$ coincides with the
right translation $x\in \sl \mapsto x\, a_2$. Since $K^P$ is right invariant,
then $K^P$ is also $l_{a_2}$-invariant. Therefore, $K^P$ is invariant under the
left action of $\Delta $, and the first sentence in the statement of the lemma follows.
We will keep the notation $K^P$ for the induced vector field\footnote{Note that
the induced vector field $K^P$ on $W$ is a Killing field for the quotient metric
$g_W$ defined by (\ref{eq:gW}), independently on the left invariant metric $g$ on
$\sl $.}  on
$W$, and $\G^P$ for the 1-parameter family of diffeomorphisms\footnote{These are isometries
of $(W,g_W)$.} of $W$ that generate $K^P$.

Now consider a properly embedded surface $\Sigma _1\subset \sl $ invariant
under the left action of $\G^P$ and under $l_{a_2}$.
We first show that given $h\in \G^H$ and $b\in \Delta $, the
left translation by $b$ leaves $l_h(\Sigma _1)$ invariant. As
$\Delta $ is generated by $a_1,a_2$, it suffices to consider the cases $b=a_1$ and
$b=a_2$.
\par
\vspace{.2cm}
(I) Assume $b=a_2$. Since $a_2$ lies in the center of $\sl $, then
$l_{a_2}(l_h(\Sigma _1))= l_h(l_{a_2}(\Sigma _1))\subset l_h(\Sigma _1)$,
because $l_{a_2}(\Sigma _1)\subset \Sigma _1$.
\par
\vspace{.2cm}
(II) Suppose $b=a_1$. Since $\G^P$ is a normal subgroup of $\Hip ^2_{\t }$ and
$\G^H\subset \Hip ^2_{\t }$, then $a_1h=ha_1'$ for some $a_1'\in
\G^P$, and thus $l_{a_1}(l_h(\Sigma _1))=l_h(l _{a_1'}(\Sigma _1))\subset l_h(\Sigma _1)$,
 because $\Sigma _1$ is invariant
under the left translation by every element in $\G^P$.
\par
\vspace{.2cm}
This proves that $\Delta $ leaves $l_h(\Sigma _1)$ invariant. The property that
$\Delta $ acts properly and discontinuously on $l_h(\Sigma _1)$ is obvious.
That the left action of $\G^P$ leaves $l_h(\Sigma _1)$ invariant follows
from the previous proof in (II), taking into account that $a_1$ can be chosen to be
any element in $\G^P-\{e\}$. Now item~(1) of the lemma is proved.

Next suppose that each integral curve of $K^H$ intersects $\Sigma_1$ in a single point
(in particular, $\Sigma _1$ is a properly embedded topological plane). The surfaces
 $l_h(\Sigma _1)$ with $h\in \G^H$ are then pairwise disjoint, properly embedded topological planes
in $\sl $ that define a product
 foliation $\mathcal{F}(\Sigma _1,\G^H)$ of $\sl $, and item~(1) implies that each leaf
$l_h(\Sigma _1)$ of $\mathcal{F}(\Sigma _1,\G^H)$
gives rise to a properly embedded quotient surface
$l_h(\Sigma _1)/\Delta \subset W$  which is invariant
under the induced action of $\G^P$ on $W$. Since the fundamental group of $l_h(\Sigma _1)$ is trivial,
then the fundamental group of $l_h(\Sigma _1)/\Delta$ is isomorphic to $\Z\times \Z$ and
so $l_h(\Sigma _1)/\Delta$
is a torus. Finally, when $h$ varies in $\G ^H$, the related tori $l_h(\Sigma _1)/\Delta$
are pairwise disjoint, thereby defining a product foliation of $W$.
\end{proof}

Recall from (\ref{eq:GHEP}) in the previous section that we parameterized the 1-parameter subgroup
$\G ^H$ by a group homomorphism
\begin{equation}
\label{eq:h}
t\in \R \mapsto h(t)=\G^H(t) \quad \mbox{ with }(\G^H)'(0)=E_2(e).
\end{equation}
In particular,
for every left invariant metric $g$ on $\sl $, the parameterized curve $t\mapsto h(t)=\G^H(t)$ is an
embedded geodesic (because its velocity vector $h'(t)$ has constant length $\l _2$ if $g$
corresponds to a triple
$(\l _1,\l _2,\l _3)\in (\R^+)^3$ by the correspondence of Proposition~\ref{metrics},
and its trace $\G^H$ is the set of fixed
points of an order-two, orientation-preserving isometry
of $X=(\sl ,g)$ around $\G ^H$). In the sequel we will identify $\G ^H$ with $\R $ through
the parameterization $h$.
When $t\to +\infty $, we produce an end point $\theta \in \partial _{\infty }\Hip ^2$
of the projection of $\G^H=h(\R)$ through
the map $\Pi \colon \sl \to \Hip ^2$ defined in (\ref{eq:Pi}). Associated to $\theta $ we have the
set $\mathcal{C}$ of horocycles in $\HH^2$ tangent to $\t$ at infinity.
Inside $\mathcal{C}$ we distinguished the horocycle $\a _0$ that passes through
$\Pi (e)\in \Hip^2$, parameterized appropriately so that $\a _0(0)=\Pi (e)$.
We called the lifted surface $\Pi^{-1}(\a_0)$ the {\it horocylinder}
over $\a _0$; see Figure~\ref{horocyl1}.

\begin{lemma}
\label{lem5.3}
The horocylinder $\Sigma _0=\Pi^{-1}(\a_0)$ satisfies the hypotheses of the
surface $\Sigma_1$ in Lemma~\ref{claim4.3}.
 Hence, the product foliation of $X$
\begin{equation}
 \label{horocfol}
\cF _0=\mathcal{F}(\Sigma _0,\G^H)=\{ l_h(\Sigma _0) \ | \ h\in \G^H\}
\end{equation}
descends to a product foliation $\cF _0/\Delta $ of $W$ by the tori $l_t(\Sigma_0)/\Delta$, $t\in \R=\G^H$,
each of which is invariant under the induced action of $\G^P$ on $W$, and $W$ is diffeomorphic
to the product of a torus with $\R $.
\end{lemma}
\begin{proof}
Consider the auxiliary left invariant metric $g_0$ on $\sl $ that makes the basis
in (\ref{eq:[]PSL2R}) orthonormal. Thus, the isometry group of $X_0=(\sl ,g_0)$
is four-dimensional and the projection
$\Pi\colon X_0\to \H^2$ given by (\ref{eq:Pi}) is, after rescaling the left invariant metric,
a Riemannian submersion,
as explained at the end of Section~\ref{secCh+}. It is worth remembering that $X_0$ is isometric
to the semidirect product $\R^2\rtimes_A \R$ with
\begin{equation}
\label{eq:Amatrix}
A=\left(\begin{array}{cc} 2 & 0 \\ 2 & 0 \end{array}\right) ,
\end{equation}
endowed with its canonical left invariant metric; to see why this property holds,
we can apply part (2) of Theorem~2.14 of~\cite{mpe11} to conclude that
$X_0$ is isometric to $\R^2\rtimes_{A(b)} \R$
where $A(b)=\left(\begin{array}{cc} 2 & 0 \\ 2b & 0 \end{array}\right) $; to see that $b=1$,
simply observe that the eigenvalues of the Ricci tensor of $X_0$ are $-6$ (double)
and $2$ (simple), while the eigenvalues of the Ricci tensor of a semidirect product with
its canonical metric are given by equation (2.23) in~\cite{mpe11}. Equality in both collections
of Ricci eigenvalues easily lead to the desired property that $b=1$.
The Riemannian three-manifold $X_0$ is commonly referred to as an $\E(\kappa ,\tau )$-space
with $\kappa =-4$ and $\tau ^2=1$; to be precise, $\Pi\colon X_0\to \HH^2$ is a Riemannian
submersion, where the usual metric on $\HH^2$ has been scaled so that it has
sectional curvature $-4$.

The reader should be aware that in spite of the fact that $X_0$ is isometric
to $\R^2\rtimes _A\R $ with its canonical metric,
the group structure on $X_0$ (that is, the one of $\sl $)
is not isomorphic to the one given by (\ref{eq:operation}) in $\R^2\rtimes _A\R $, as follows from the fact that
$\sl $ is unimodular while $\R^2\rtimes _A\R $ is non-unimodular with the matrix
$A$ above.
Nevertheless, the non-isomorphic Lie groups $\sl $ and $\R^2\rtimes _A\R $ can be considered as
three-dimensional subgroups of the four-dimensional isometry group Iso$(X_0)$ of $X_0$,
both acting by left translation on Iso$(X_0)$ with the same identity element equal to $1_X\in \mbox{Iso}(X_0)$.
In this setting, Corollary~3.19 in~\cite{mpe11} ensures that the connected component
of $[\sl \cap (\R^2\rtimes _A\R )]\subset \mathrm{Iso}(X_0)$ passing through $1_X$ is the two-dimensional
subgroup $\mathcal{H}$ of $\R^2\rtimes _A\R $ given by
\[
 \cH =\{ (x,x,z)\ | \ x,z\in \R \} ,
 \]
 and the entire intersection $[\sl \cap (\R^2\rtimes _A\R )]\subset \mathrm{Iso}(X_0)$ is generated by $\cH$ and the center of $\sl$.
Viewed inside $\sl $, $\mathcal{H}$ corresponds to one of the non-commutative subgroups $\Hip ^2_{\mu }$
given by (\ref{eq:Htheta}) for some $\mu \in \partial _{\infty }\Hip ^2$.
After conjugating the embedding of $\R^2\rtimes_A \R$ into $\mathrm{Iso}(X_0)$ by an appropriate isometry of
$X_0$ which is a rotation around $\G^E$,
we can identify $\mathcal{H}$ inside $\sl $ with the
two-dimensional subgroup $\Hip ^2_{\theta }$ defined just before Definition~\ref{defhorocyl}.
Therefore, geometric and algebraic objects inside $\mathcal{H}\subset \R^2\rtimes _A\R $
have their counterparts inside $\Hip ^2_{\theta }\subset \sl $; for instance, the unique parabolic 1-parameter normal subgroup
$\{ (x,x,0)\ | \ x\in \R \} $ of $\mathcal{H}$ corresponds to $\G ^P$ inside $\Hip ^2_{\theta }$,
and the hyperbolic 1-parameter subgroup $\{ (0,0,z)\ | \ z\in \R \} $ (which is a unit speed
geodesic of $\R^2\rtimes _A\R $ parameterized by $z\in \R $) corresponds to the hyperbolic
1-parameter subgroup $\G^H$ inside $\Hip ^2_{\theta }$.

Under the above isometric identification of $\sl$ with $\R^2\rtimes _A\R $,
the horocylinder $\Sigma _0\subset \sl $, which has constant mean curvature in $X_0$,
corresponds to the horizontal plane $\R^2\rtimes _A\{ 0\} $
(note that $\R^2\rtimes _A\{ 0\} $ is a subgroup of $\R^2\rtimes _A\R $ but
$\Sigma _0$ is not a subgroup of $\sl $), and the center of $\sl$, which is contained
in $\Sigma _0$, can be viewed as a discrete subset of  $\R^2\rtimes_A \{0\} \subset \R^2\rtimes _A\R $.
From here is not difficult to conclude that the $(\Z \times \Z )$-subgroup $\Delta $ of $\sl $ corresponds
to a $(\Z \times \Z)$-lattice $\wt{\Delta }$ inside the two-dimensional subgroup
$\R^2\rtimes_A \{0\}$ of $\R^2\rtimes _A\R $. As a consequence, the quotient space $W=\sl /\Delta$ is
diffeomorphic to $(\R^2\rtimes _A\R )/\wt{\Delta }$, which is diffeomorphic to the product of
the two-torus $(\R^2\rtimes _A\{ 0\} )/\wt{\Delta }$ with the real line $\{ (0,0,z)\ | \ z\in \R \} $.
Now the properties in the statement of the lemma follow easily from the corresponding properties
in $\R^2\rtimes _A\R $; for instance, the foliation $\cF_0=\mathcal{F}(\Sigma _0,\G^H)$ corresponds to
$\{ \R^2\rtimes _A \{ z\} \ | \ z\in \R \} $. This completes the proof.
\end{proof}

We will use the notation introduced in the proof of the previous lemma to make an observation that will be
useful later.
\begin{lemma}
\label{lemNhoroc}
 Consider the left invariant metric $g_0$ on $\sl $ that makes the basis in (\ref{eq:[]PSL2R})
 orthonormal. Let $N_{\cF _0}$ be the unit normal vector field to the foliation $\cF _0$
 of $X_0=(\sl ,g_0)$  by flat horocylinders  defined in (\ref{horocfol}). Then:
 \begin{equation} \label{eq:5.19}
N_{\cF _0}= K^H+f\, K^P
\end{equation}
for  some smooth function $f\colon \sl \to \R $.
\end{lemma}
\begin{proof}
As described in the proof of Lemma~\ref{lem5.3}, $\G^H$ and $\G^P \subset \sl$ correspond, respectively,
 to the 1-parameter
subgroups $\{ (0,0,z)\ | \ z\in \R \} $ and $\{ (x,x,0)\ | \ x\in \R \} $ of $\R^2\rtimes _A\R $.
Keeping the same notation as in the proof of Lemma~\ref{lem5.3},
let
\[
F_3=2x\, \partial _x+2x\, \partial _y+\partial _z, \quad
\wh{K}^P=\partial _x+\partial _y
\]
be the right invariant vector fields on $\R^2\rtimes _A\R $ generated by the right actions
of $\{ (0,0,z)\ | \ z\in \R \} $ and  $\{ (x,x,0)\ | \ x\in \R \} $, respectively.
Then we have
\[
K^H=F_3=2x\, \wh{K}^P+
\partial _z .
\]
As the unit normal vector field to the foliation $\{ \R^2\rtimes _A \{ z\} \ | \ z\in \R \} $
of $\R^2\rtimes _A\R $ is $\partial _z$, then
the last equality implies that (\ref{eq:5.19}) holds with
$f=2x\colon
\R^2\rtimes _A\R \to \R$, which can be considered to be a smooth function on $\sl $.
\end{proof}

We now collect some properties involving any prescribed left invariant metric $g$ on $\sl $ and its
locally homogeneous quotient metric $g_W$ on $W$.
\begin{lemma}
\label{lem5.3bis}
Let $E_W$ be the end of $W$ that contains the proper arc $\pi_W(\Gamma^H [0,\8))$,
where $\G^H$ is parameterized as a subgroup by $h\colon \R \to \G^H$ defined in (\ref{eq:h}).
Given a left invariant metric $g$ on $\sl $, the following properties hold:
\ben
\item The locally homogenous manifold
$(W,g_W)$ has infinite volume and the volume with respect to $g_W$
of every end representative of $E_W$ is finite.

\item If $\Sigma_1\subset X$ is any surface satisfying the hypotheses of Lemma~\ref{claim4.3}-(2),
 then  the area function of the related quotient tori,
 \[
t\in \R \mapsto  A(t)=\mbox{\rm Area}\left( l_{h(t)}(\Sigma_1)/\Delta, g_W\right)
  \]
is exponentially decreasing as $t\to +\infty $. Furthermore, an end
 representative of $E_W$ can be chosen to be
 $\bigcup _{t\in [0,\infty)} [l_{h(t)}(\Sigma_1)/\Delta ]$.
\een
\end{lemma}
\begin{proof}
Since any two left invariant metrics on a metric Lie group are quasi
isometric by the identity mapping,
it suffices to prove the result for any particular left invariant metric on $\sl $.
So, choose the metric $g=g_0$ on $\sl $ to be the one that makes the basis
in (\ref{eq:[]PSL2R}) orthonormal.

As explained in the proof of Lemma~\ref{lem5.3}, $X_0=(\sl ,g_0)$ is isometric to
$\R^2\rtimes _A\R $ endowed with its canonical metric ($A$ is the matrix given by
(\ref{eq:Amatrix})) and the foliation
$\{ \R^2\rtimes _A\{ z\} \ | \ z\in \R \} $ of $\R^2\rtimes_A \R$, all whose leaves have the same
constant mean curvature, corresponds to the foliation $\cF _0$ of $X_0$ by parallel horocylinders
defined in (\ref{horocfol}). Recall from the proof of Lemma~\ref{lemNhoroc} that
the unit vector field $\partial _z$ in the model
$\R^2\rtimes_A \R$ of $X_0$ corresponds to the unit normal field to the leaves of $\cF(\Sigma _0,\G^H)$.

By formula (3.51) in \cite{mpe11},
the volume element associated to the canonical metric on $\R^2\rtimes_A \R$ is
 \begin{equation}
 \label{volelem}
  {\rm d Vol} = e^{-z \,{\rm trace} (A)} dx \wedge dy \wedge dz.
  \end{equation}
Equation (\ref{eq:13}) implies that
$\partial _z$ has constant length 1 and is everywhere orthogonal to the vectors fields
$\partial_x$, $\partial _y$, whose lengths and inner product are given by
 \begin{equation}
 \label{lengths}
 \| {\textstyle \partial_x}\| ^2=e^{-4z}+(e^{-2z}-1)^2,\quad
 \| {\textstyle \partial_y}\| =1,
 \quad \langle {\textstyle \partial_x},{\textstyle \partial_y}\rangle =
 e^{-2z}-1.
 \end{equation}
 Now consider the $(\Z \times \Z )$-lattice $\wt{\Delta }\subset \R^2\rtimes _A\{ 0\} $ corresponding
 to the subgroup $\Delta \subset \sl $ through the isometry between $X_0$ and $\R^2\rtimes _A\R $
 (see the end of the proof of Lemma~\ref{lem5.3}). Equation (\ref{volelem}) implies
that the end of $(\R^2\rtimes _A\R )/\wt{\Delta }$ in which the geodesic arc $\{(0,0)\}\times [0,\8)$
lies has finite volume with respect to the quotient of the canonical metric on $\R^2\rtimes _A\R $
(and the other end has infinite volume), hence the same properties hold for $(W,g_W)$.
It also follows that the area
element of the plane $\R^2\rtimes_A \{z\}$ has the form $d{\rm Area}=e^{-z \,{\rm trace} (A)} dx \wedge dy$,
which implies that the areas of the tori $(\R^2\rtimes _A\{ z\} )/\widetilde{\Delta }$ in the related foliation
of $(\R^2\rtimes _A\R )/\wt{\Delta }$ decrease exponentially as a function of $z\to +\infty $.
In fact, if $\Sigma _1\subset X$ is a surface satisfying the hypotheses in item~(2) of the lemma, then
$\Sigma _1$ can be considered to be a doubly-periodic graph over $\R^2\rtimes _A\{ 0\} $, from where
we directly deduce that the $(g_0)_W$-area of $l_{h(t)}\Sigma /\Delta $ decreases exponentially as
$t\to +\infty $ (since the tori $(\R^2/\Delta )\rtimes _A\{ t\} $ have the same property).
This completes the proof of the lemma.
\end{proof}

The main result of this section is the following one:
\begin{proposition}
  \label{claim:3d}
Let $X$ be $\sl$ endowed with a left invariant metric $g$. 
Then, $2 H(X) = {\rm Ch}(X)$.
Furthermore, there exists a properly embedded surface $\Sigma
\subset X$ with constant mean curvature $H(X)$ such that the
following properties hold.
\begin{enumerate}[(A)]
\item $\G^P\subset \Sigma $ and $K^P$ is everywhere tangent to $\Sigma $
(equivalently, $\Sigma $ is invariant under the left translation by
every element in $\G^P$ and $e\in \Sigma$).

\item $\Sigma $  intersects each of the integral curves of $K^H$
transversely in a single point. Equivalently, $\Sigma $ is
an entire graph with respect to the Killing vector field $K^H$ on $X$ generated by $\G^H$.
In particular, $\Sigma $ is a topological plane and
\[
\mathcal{F}= \mathcal{F}(\Sigma ,\G ^H)=\{ l_h(\Sigma )\ | \ h\in \G^H\}
\]
is a product foliation of $X$, all whose leaves have constant
mean curvature $H(X)$.

\item $\Sigma $ is invariant under the left translation $l_{a_2}$. Furthermore, the
$(\Z \times \Z )$-subgroup $\Delta $ of isometries of $X$ appearing in
Definition~\ref{def:basis} acts properly and discontinuously on $\Sigma $, thereby defining a
quotient surface $\Sigma /\Delta$ in $(W=\sl /\Delta ,g_W)$.

\item Each leaf of the foliation  $\mathcal{F}$ is invariant under $\Delta$, and
$\mathcal{F}$ descends to a product quotient
foliation $\mathcal{F}/\Delta $ of $(W,g_W)$ by tori with constant mean
curvature $H(X)$.

\item Given $T\in \R$, consider the domain
\begin{equation}
\label{end}
\mathcal{D}(T)=\bigcup _{t=T}^{\infty }[l_{h(t)}(\Sigma )/\Delta ],
\end{equation}
which is an end representative of the end $E_W$ of $(W,g_W)$ with finite volume.
Then, $\mathcal{D}(T)$ is the unique solution to the isoperimetric problem in $(W,g_W)$ for
its (finite) value of the enclosed volume. Moreover, the mean curvature vector
of $\parc \cD(T)$ points into $\cD(T)$.

\end{enumerate}
\end{proposition}

Before proving the above proposition, we state and prove the following key lemma whose
proof will occupy several pages of this section.

\begin{lemma}\label{lem:H}
Let $X$ be $\sl $ endowed with a left invariant metric. Suppose $\Sigma \subset X$
is a properly embedded surface of
constant mean curvature $H \geq 0$ that satisfies conditions (A), (B) and (C) in
Proposition~\ref{claim:3d}. Then, $2H={\rm Ch}(X)=2 H(X)$ and $\Sigma$ satisfies
the remaining properties (D) and (E) in Proposition~\ref{claim:3d}.
\end{lemma}
\begin{proof}
By Lemma~\ref{claim4.3}, $\Delta $ acts properly and discontinuously
on every leaf of the product foliation $\cF= \mathcal{F}(\Sigma ,\G ^H)$, and $\cF $ descends
to a quotient product foliation $\mathcal{F}/\Delta =\{ l_{h(t)}(\Sigma )/\Delta \ | \ t\in \R \} $
of $W$ by tori. In particular, $W$ is diffeomorphic to $(\Sigma/\Delta)\times \R$.
As $\Sigma $ has constant mean curvature $H$ in $X$ and $\Delta \subset \mbox{Iso}(X)$,
then $\Sigma /\Delta $ has constant mean curvature $H$ in $(W,g_W)$
and the same holds for all the leaves of $\mathcal{F}/\Delta $ because these leaves are
all isometric by an ambient isometry.
 So, in order to show that $\Sigma$ satisfies condition~(D) of
Proposition~\ref{claim:3d},
it remains to prove that $H=H(X)=\frac{1}{2} {\rm Ch}(X)$, which we do next.

By Lemma~\ref{lem5.3bis}, the end $E_W$ containing the proper arc $\pi _W(\G^H [0,\infty ))$
has finite volume with respect to $g_W$. Recall that we defined in (\ref{end})
the end representative $\cD(T)$ of $E_W$, for each $T\in
\R $. It is worth reparameterizing $T\mapsto \cD(T)$ by its enclosed finite volume
with respect to $g_W$: as $T\in \R \mapsto \mbox{Vol}(\cD (T),g_W)$ is strictly decreasing,
we can define for each $V>0$ the end representative
\[
\Omega  (V):=\cD(T(V))=\bigcup _{t=T(V)}^{\infty }[l_{h(t)}(\Sigma  )/\Delta ],
\]
of $E_W$ where $T(V)$ is uniquely defined by the equality {\rm Vol}$(\Omega
 (V),g_W)=V$.

\begin{assertion}
\label{claim5.6}
In the above situation,
 $H\neq 0$ and given $V>0$, the mean curvature vector of $\parc \Omega  (V)$ 
 points into $\Omega  (V)$. Equivalently, for every $T\in \R$, the mean curvature vector
of $\parc \cD(T)$ points into $\cD(T)$.
\end{assertion}
\begin{proof}
We will apply the Divergence Theorem to the unit normal field
$N_{{\cF }/\Delta }$ of the foliation $\mathcal{F}/\Delta $ (with respect to the
metric $g_W$). As the leaves of $\cF /\Delta $ are the quotient tori of left
translations of $\Sigma $ by elements $h\in \G^H$, then the unit normal vector
$N_{\cF}$ to $\mathcal{F}$ satisfies
\[
N_{\cF}(l_h(q))=(l_h)_*(N_{\Sigma }(q)), \quad \mbox{for every $q\in \Sigma $, $h\in \G^H$,}
\]
where $N_{\Sigma }$ stands for the unit normal vector to $\Sigma $ (we will
assume without loss of generality that $N_{\Sigma }$ is the unit normal for which the
mean curvature of $\Sigma $ is $H\geq 0$).
Furthermore, the divergence of $N_{\cF}$ (resp. of $N_{\cF /\Delta }$) with
respect to $g$ (resp. to $g_W$) is equal  to the negative
of twice the mean curvature $H$ of the leaves of $\mathcal{F}$ (resp. of
$\mathcal{F}/\Delta $).
Since $\Omega  (V)$ has finite volume and $\mbox{div}_W\left( N_{\cF /\Delta }\right)$ has a fixed sign,
the Divergence Theorem gives
\[
-2H\, V=-2H\, {\rm Vol} (\Omega (V),g_W) =\int _{\Omega (V)}\mbox{div}_W\left( N_{\cF /\Delta }\right) =
\int_{\parc \Omega (V)} g_W( N_{{\cF }/\Delta },N_{\parc \Omega (V)}),
\]
where $N_{\parc \Omega (V)}$ is the outward pointing
unit normal of $\parc \Omega (V)$. Since $\partial \Omega (V)=l_{h(T(V))}(\Sigma )/\Delta $, then
$(N_{{\cF }/\Delta })|_{\parc \Omega (V)}=\ve N_{\parc \Omega (V)}$ where $\ve =\pm 1$ and
the last equation reads $-2H\, V=\ve \mbox{Area}\left( l_{h(T(V))}(\Sigma )/\Delta ,g_W\right) $. In particular,
$H > 0$ and $\ve =-1$. Finally, the mean curvature vector of $\parc \Omega  (V)$ is
$H(N_{{\cF }/\Delta })|_{\parc \Omega (V)}=-H N_{\parc \Omega (V)}$, which
 points into $\Omega  (V)$.
\end{proof}

Given $q\in X$, let $\g _q\colon \R \to X$ be the integral curve of
$N_{\cF}$ passing through $q$, i.e.
$\g _q(0)=q$, $\g _q'(u)=N_{\cF}(\g _q(u))$, for all $u\in \R $
(note that $\g _q$ is defined for every value of $u\in \R $ as $X$ is
complete and $N_{\cF}$ is bounded). Let $\phi_u\colon X\to X$, $\phi _u(q)=
\g _q(u)$, for all $u\in\R$. Thus, $\{ \phi _u\} _{u\in \R }$ is the
1-parameter group of diffeomorphisms of $X $ generated by $N_{\cF}$.

\begin{assertion}
\label{claim2}
 There exist constants $C_1,C_2>1$ such that:
\ben
\item For any unit tangent vector $v_p \in T_pX$ at a point $p\in X$ and $u\in [0,\infty]$,
$\frac{1}{C_1e^u}\leq |(\phi_u)_{*}(v_p)|\leq C_1e^{u}$.
\item Recall that $\G^H$ was considered to be parameterized
by $h(t)\in \G^H$, $t\in \R $ (equation (\ref{eq:h})). Then for any point
$q\in \Sigma  $ and $t_0\geq 1$, it holds that $\g_q(C_2 t_0)$
 lies in the set
$\bigcup_{t\in [t_0,\infty)} l_{h(t)}(\Sigma )$.
\een
\end{assertion}
\begin{proof}
Recall that $\Sigma $ is invariant under the left action of $\Delta $, with quotient a torus
$\Sigma /\Delta $.  Since left translations by elements in
$\Delta \cup \G^H$ leave invariant $\cF  $, then $N_{\cF}$ is also
invariant under left translations by elements in $\Delta \cup \G^H$.
This property together with the invariance of $\Sigma $ under the left
action of $\Delta $
imply that $|(\phi _u)_*(v_p)|$ is uniformly bounded and bounded away from zero
independently of $u\in [0,1]$, $p\in X$ and $v_p\in T_pX$ with $|v_p|=1$.
A straightforward iteration argument in the variable $u$
gives the estimate in item~(1) of Assertion~\ref{claim2}.

To prove item~(2), consider the function $\wt{f}\colon X \to \R $ defined by
\begin{equation}
\label{eq:5.2bis}
\wt{f}(p)=t \quad \mbox{ if }\quad p\in l_{h(t)}(\Sigma  ).
\end{equation}
Clearly, $\wt{f}$ is invariant under the
left action of $\Delta $ as each left translate of $\Sigma $ has the same property.
Take a point $p\in X$ and call $t=\wt{f}(p)$. Given $s\in \R $, $l_{h(s)}(p)\in
(l_{h(s)}\circ l_{h(t)}) (\Sigma )=
l_{h(s+t)}(\Sigma )$ because $t\mapsto h(t)$ is a group homomorphism. Thus,
\begin{equation}
\label{eq:5.2}
  \wt{f}\circ l_{h(s)}=s+\wt{f},\quad \forall s\in \R.
\end{equation}
A simple consequence of (\ref{eq:5.2}) is that
the gradient $\nabla  \wt{f}$ of $\wt{f}$ with respect to the metric $g $
satisfies $(\nabla \wt{f})\circ l_{h(s)}=(l_{h(s)})_*(\nabla \wt{f})$, i.e.,
$\nabla \wt{f}$ is invariant under the left action of $\G^H$.
$\nabla \wt{f}$ is also invariant under the left action of $\Delta $,
as $\wt{f}$ has the same property. Observe that $\nabla \wt{f}$ is orthogonal to
$l_{h(t)}(\Sigma  )$, hence
\begin{equation}
\label{eq:grdwtf}
\nabla \wt{f}=|\nabla  \wt{f}| \, N_{\cF}.
\end{equation}
Another simple consequence of (\ref{eq:5.2}) is that
$\nabla \wt{f}$ is nowhere zero on $X$ (take the derivative with respect to $s$).

As $\nabla  \wt{f} $ descends without zeros to the torus $[l_{h(t)}(\Sigma )]/\Delta $ for any $t$
, then $\nabla \wt{f}$ is bounded by above and below by some positive
constants in every leaf $l_{h(t)}(\Sigma )$ of $\cF$. Since $\nabla \wt{f}$
is also invariant under the action of $\G^H$, we deduce that
\begin{equation}
\label{eq:5.3}
C\leq |\nabla \wt{f}|\leq 1/C\quad \mbox{ in $X$, \quad for some $C\in (0,1)$.}
\end{equation}
Let
$(q,u)\in \Sigma \times [0,\infty )\mapsto t=t(q,u)$ be the smooth function such that
$\g _q(u)\in l_{h(t)}(\Sigma )$. Then, $t(q,u)=\wt{f}(\g _q(u))$ and
\[
\frac{\partial t}{\partial u}(q,u)=\frac{d}{du}(\wt{f}\circ \g _q)(u)=
g(\nabla \wt{f},\g _q'(u))(\g _q(u))
\]
\[
=g(\nabla \wt{f},N_{\cF})(\g _q(u))=|\nabla \wt{f}|(\g _q(u))\geq C ;
\]
hence after integration with respect to $u$,
\begin{equation}
\label{eq:t(q,u)}
t(q,u)\geq Cu-t(q,0)=Cu,\quad \mbox{for all $u\geq 0$, $q\in \Sigma $.}
\end{equation}
Now define $C_2=1/C$. Given $t_0\geq 1$, $q\in \Sigma $, taking $u=C_2t_0$ in~(\ref{eq:t(q,u)})
we have $t(q,C_2t_0)\geq t_0$, which means that $\g _q(C_2t_0)\in
\bigcup_{t\in [t_0,\infty)} l_{h(t)}(\Sigma )$, as desired. This finishes the
proof of Assertion~\ref{claim2}.
\end{proof}

 The fact that $\Sigma  $ is invariant under the left action of $\Delta $
 (by isometries of $g$) clearly implies the following statement.
\begin{assertion}
\label{claim3a}
Let $\sigma_1,\sigma_2$ be least-length simple
closed geodesics in the torus $\Sigma/\Delta$, that can be considered to be
generators of the first homology group of $\Sigma/\Delta$. Let $L_i>0$ be
the length of $\sigma _i$, $i=1,2$. For each
$n\in\N$, let $F(n)\subset\Sigma  $ be a 'square' collection of $n^2$
adjacent fundamental domains for the action of the group $\Delta $, so that
$\partial F(n)$ consists of four geodesic arcs in $\Sigma$, each of which
covers $n$ times one of the geodesics $\sigma_1,\sigma_2$. Then,
\begin{equation}\label{length}
\mathrm{Length}(\partial F(n),g)= C_3 n, \qquad
\mathrm{Area}(F(n),g)=C_4 n^2,
\end{equation}
where $C_3=\mbox{\rm Length}(\partial F(1))=2(L_1+L_2)$, and
$C_4=\mbox{\rm Area}(F(1))$.
\end{assertion}

\begin{assertion}
\label{claim4.6bis}
  In the above situation, $H =H(X )=\frac{1}{2}\mbox{\rm Ch}(X )$
  (and thus, $\Sigma $ satisfies condition (D) of Proposition~\ref{claim:3d}).
\end{assertion}
\begin{proof}

 Given a compact, orientable smooth surface $S$ immersed in $X $,
the existence of the foliation $\mathcal{F}=\cF (\Sigma ,\G^H) $ of $X$ by surfaces of constant mean
curvature $H >0$ and an
elementary comparison argument for the mean curvature shows that the maximum of the
absolute mean curvature function of $S$ (with respect to the metric
$g $) is at least $H $. Then, by definition of critical mean
curvature we have $H \leq H(X )$. By Lemma~\ref{lem2.4}, we have
$H(X )\leq \frac{1}{2}\mbox{Ch}(X )$; hence to finish the proof of
the Assertion it suffices to show that $\mbox{Ch}(X )\leq 2H $.
To do this, given any $\ve>0$ we will construct a piecewise smooth compact
domain $B$ of $X $ such that the ratio of the area of the boundary $\partial B$
to the volume of $B$ is bounded from above by $2H +\ve$. The inequality
Ch$(X )\leq 2H  +\ve$ will then follow from the definition of
Ch$(X )$. As $\ve>0$ is arbitrary in this construction, we will deduce the desired
estimate, thereby proving Assertion~\ref{claim4.6bis}. The construction of
the domain $B$ is motivated by a similar construction
for estimating the Cheeger constant of certain metric Lie groups
in the proof of the main theorem of Peyerimhoff and  Samiou in~\cite{pesa1};
also see  Section~3.8 in~\cite{mpe11} for these similar constructions.

For each $n\in\N$ and $t_0\geq 1$, consider the following domain:
\[
B(n,t_0)=\left(\bigcup_{t\in[0,t_0]} l_{h(t)}(\Sigma  )\right)
\bigcap\left(\bigcup_{u\in \R}\phi_u(F(n))\right).
\]
The boundary of the box-shaped solid region $B(n,t_0)$ consists of
the 'bottom square' face $F(n)$, the 'top square' face
$F^{\mbox{\footnotesize top}}(n,t_0)=[l_{h(t_0)}(\Sigma  )]\bigcap
[\bigcup_{u\in \R}\phi_u(F(n))]$ and the 'sides of the box', which is
the piecewise smooth surface
$S(n,t_0)=\partial B (n,t_0)- [F(n)\cup F^{\mbox{\footnotesize
top}}(n,t_0)]$. By item~(2) of Assertion~\ref{claim2}, $S(n,t_0)$ is contained in
the piecewise smooth surface
\[
\wt{S}(n,t_0)=\bigcup_{0\leq u\leq C_2t_0}\phi_u(\partial F(n)),
\]
see Figure~\ref{figbox}.
\begin{figure}
\begin{center}
\includegraphics[width=10cm]{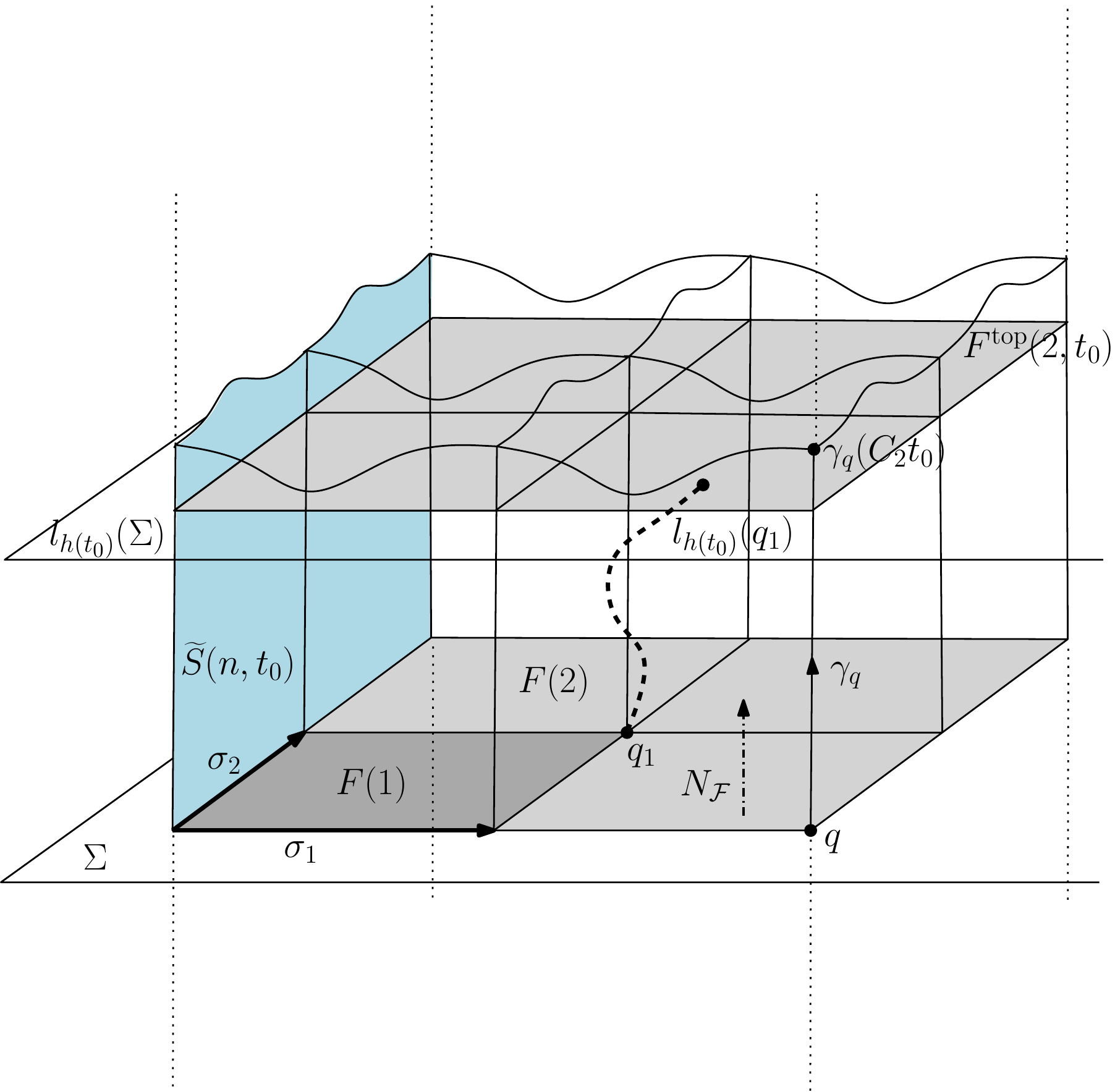}
\caption{Schematic representation of a 'box'-shaped domain $B(n,t_0)$ for $n=2$, obtained by
intersecting the 'slab' $\bigcup_{t\in[0,t_0]} l_{h(t)}(\Sigma  )\subset X$ bounded
by $\Sigma $ and its left translate $l_{h(t_0)}(\Sigma )$, with the vertical 'square
column' $\bigcup_{u\in \R}\phi_u(F(n))$ produced by moving the square $F(n)\subset \Sigma $
under the flow $\{ \phi _u\} _u$ of the unit normal vector field $N_{\cF }$. The bottom and top
faces $F(n)$ and $F^{\mbox{\footnotesize top}}(n,t_0)$ of $B(n,t_0)$ are represented in grey.
The side faces $S(n,t_0)$ of $B(n,t_0)$ are contained in larger side faces $\wt{S}(n,t_0)$
(in blue) obtained by flowing $\partial F(n)$ by $\phi _u$, $0\leq u\leq C_2t_0$. The dashed
curve represents the trajectory of a given point $q\in \Sigma $ under left multiplication
by elements in $\G^H$, until reaching $l_{h(t_0)}(\Sigma )$; this trajectory might differ
from the integral curve $\g _q$ of $N_{\cF }$.}
 \label{figbox}
\end{center}
\end{figure}

We next obtain an estimate for the area of the 'top' face
$F^{\mbox{\footnotesize top}}(n,t_0)$.
%
%
%
Since $\Sigma $ satisfies  the hypotheses of Lemma~\ref{claim4.3}-(2), then
Lemma~\ref{lem5.3bis}-(2) ensures that the area function (with respect to $g_W$) of the related quotient tori
$t\mapsto l_{h(t)}(\Sigma)/\Delta$ is exponentially decreasing as $t\to +\infty $. Therefore, by (\ref{length}),
there is some $C_5>0$ such that
\begin{equation} \label{top}
\mathrm{Area}(F^{\mbox{\footnotesize top}}(n,t_0))<C_5 n^2e^{-t_0}.
\end{equation}

We will also need an upper estimate for the area of the 'sides' of $B(n,t_0)$. To do this, consider the function
$\wt{h}\colon \wt{S}(n,t_0)\to [0,C_2t_0]$ given by $\wt{h}(\phi _u(q))=u$, for each $q\in \partial F(n)$.
As the gradient $\nabla \wt{h}$ is clearly orthogonal to $\phi _u(\partial F(n))$, then
\[
|\nabla \wt{h}|(\g _q(u))=g(\nabla \wt{h},\g _q'(u))=[\g _q'(u)](\wt{h})=\frac{d}{du}(\wt{h}\circ \g _q)=1,
\]
which by the coarea formula, gives that
\[
\mathrm{Area}({S}(n,t_0))\leq
\mathrm{Area}(\wt{S}(n,t_0) )
=\int _0^{C_2t_0}\left( \int _{\phi _u(\partial F(n))}ds_u\right) du
\]
\[
=\int _0^{C_2t_0}\mbox{length}\left(
\phi _u(\partial F(n)\right) du\stackrel{\mbox{\tiny (a)}}{\leq }
\int _0^{C_2t_0}C_1e^u \mbox{length}\left( \partial F(n)\right) du
\]
\begin{equation} \label{sides}
\stackrel{(\ref{length})}{\leq }C_1C_3n\int _0^{C_2t_0}e^u\, du
\leq
C_1e^{C_2t_0}\cdot C_3n,
\end{equation}
where $ds_u$ denotes the length element in $\phi _u(\partial F(n))$
and in (a) we have used the upper bound
for $|(\phi _u)_*(v_p)|$ in item~(1) of Assertion~\ref{claim2}.

We finally prove that for any $\ve>0$, there exists a $T_0\geq 1$,
such that for any $t_0\geq T_0,$ and $n$ sufficiently large, the
next inequality holds:
\begin{equation} \label{area-top}
\frac{\mathrm{Area}(\partial
B(n,t_0))}{\mathrm{Vol}(B(n,t_0))}\leq 2H +\ve.
\end{equation}
To see that this property holds, fix some $t_0\geq 1$ and apply the
Divergence Theorem to the vector
field $N_{\cF }$ in the domain $B(n,t_0)$:
\[
-2H \, \mbox{Vol}(B (n,t_0))=\int_{B(n,t_0)}
\div_X(N_{\cF})=\int_{\partial B(n,t_0)} g(N_{\cF}, \eta )
\]
\[
=\int_{F(n)} g(N_{\cF}, \eta ) +\int_{F^{\mbox{\tiny top}}(n,t_0)}
g( N_{\cF}, \eta )+\int_{{S}(n,t_0)} g(N_{\cF}, \eta ),
\]
where $\eta$ is the outward pointing unit normal to $B(n,t_0)$ along its boundary.
Observe that $\eta =-N_{\cF }$ along $F(n)$ and $\eta =N_{\cF}$ along
$F^{\mbox{\footnotesize top}}(n,t_0)$. As $g(N_{\cF},\eta ) \leq
1$, then we obtain
\begin{equation}
\label{claim:divergence}
-2H \, \mbox{Vol}(B (n,t_0)) \leq
- \mathrm{Area}(F(n))
+\mathrm{Area}(F^{\mbox{\footnotesize top}}(n,t_0)) +
\mathrm{Area}({S}(n,t_0))
\end{equation}
\[
 = -\mathrm{Area}(\partial B(n,t_0))
+2\mathrm{Area}(F^{\mbox{\footnotesize top}}(n,t_0)) +2 \mathrm{Area}({S}(n,t_0)),
\]
which we rewrite as
\begin{equation}
\label{div}
\frac{\mathrm{Area}(\partial B(n,t_0))}{\mathrm{Vol}(B(n,t_0))}\leq
2H  +\frac{2\mathrm{Area}(F^{\mbox{\footnotesize
top}}(n,t_0))}{\mathrm{Vol}(B(n,t_0))} +
\frac{2\mathrm{Area}({S}(n,t_0)) }{\mathrm{Vol}(B(n,t_0))}.
\end{equation}

Now,
\[
\frac{2\mathrm{Area}(F^{\mbox{\footnotesize top}}(n,t_0))}{\mathrm{Vol}(B(n,t_0))} =
\frac{2\mathrm{Area}(F^{\mbox{\footnotesize top}}(n,t_0))}{n^2\mathrm{Vol}(B(1,t_0))}
\stackrel{(\ref{top})}{<}
\frac{2C_5e^{-t_0}}{\mathrm{Vol}(B(1,t_0))}\leq
\frac{2C_5e^{-t_0}}{\mathrm{Vol}(B(1,1))},
\]
which can be made less than $\ve /2$ by taking $t_0\geq 1$ large
enough. For this value of $t_0$ fixed, we have
\begin{equation}  \label{div3}
\frac{2\mathrm{Area}({S}(n,t_0))}{\mathrm{Vol}(B(n,t_0))}
\stackrel{(\ref{sides})}{\leq }\frac{2\cdot C_1e^{C_2t_0}\cdot C_3n}
{n^2\mathrm{Vol}(B(1,t_0))},
\end{equation}
and so the left hand side of (\ref{div3}) can be also made less than
$\ve /2$ by taking $n$ sufficiently large.
With these two estimates, (\ref{div}) implies (\ref{area-top}), and the proof of Assertion~\ref{claim4.6bis} is
complete.
\end{proof}

To finish the proof of Lemma~\ref{lem:H}, it remains to prove that if $\Sigma $
satisfies the hypotheses of Lemma~\ref{lem:H}, then item~(E) in the statement
of Proposition~\ref{claim:3d} holds. As the volume of the end
$\mathcal{D}(0)=\bigcup _{t=0}^{\infty }[l_{h(t)}(\Sigma )/\Delta ]$
of $W$ is finite by Lemma~\ref{lem5.3bis}-(2)
and since for every $T\in \R$ the mean curvature vector
of $\parc \cD(T)$ points into $\cD(T)$ (Assertion~\ref{claim5.6}), we only need to prove the next assertion.

\begin{assertion}
\label{ass5.10}
Let $\Sigma \subset X$ be a surface satisfying the hypotheses of Lemma~\ref{lem:H}.
Given $T\in \R$, the domain $\mathcal{D}(T)$ defined in (\ref{end})
is the unique solution to the isoperimetric problem in $(W,g_W)$ for
its (finite) value of the enclosed volume.
\end{assertion}
\begin{proof}
 Given $T\in \R$, let $V(T)$ be the volume
of $\mathcal{D}(T)$ in $(W,g_W)$. By Assertion~\ref{claim4.6bis},
the unit normal field $N_{\cF /\Delta }$ to the foliation $\mathcal{F}/\Delta $
(with respect to the metric $g_W$) has divergence equal to the negative of twice the
mean curvature $H=H(X)$ of the leaves  of $\mathcal{F}/\Delta $. By
Assertion~\ref{claim5.6}, $N_{\cF /\Delta }$ restricted to $\parc \cD(T)$
 points into $\cD(T)$.  Consider a smooth, possibly non-compact domain $\Omega
\subset W$ with volume $V(T)$, $\Omega \neq \cD (T)$. Thus,
there is a small disk $D$ in the boundary $\partial \Omega$ and an $\ve>0$ such that
\begin{equation}
\label{eq:5.12}
-\int _{D }g_W(N_{\cF /\Delta },N_{\partial \Omega })\leq \mbox{Area}(D,g_W)-\ve,
\end{equation}
where $N_{\partial \Omega }$ is the outward pointing unit normal
vector field to $\Omega $ along its boundary.

Consider the smooth function $f\colon W\to \R $ given by
\begin{equation}
\label{eq:f}
f(x)=t \mbox{ \; provided that \;} x\in l_{h(t)}( \Sigma )/{\Delta },
\end{equation}
(compare with (\ref{eq:5.2bis})). By (\ref{eq:grdwtf}) and (\ref{eq:5.3}), the gradient of $f$ in $W$ satisfies
\begin{equation}
\label{eq:gradf}
\nabla f=|\nabla f|\, N_{\cF /\Delta },\quad
C \leq |\nabla f|\leq  1/C
\end{equation}
for some $C\in (0,1)$.
Since $\Omega$ has finite volume, the coarea formula can be applied to $f$ on $\Omega$ and
gives
\[
\mbox{Vol}(\Omega ,g_W)=
\int _{-\infty }^{\infty }\left( \int _{\Omega \cap [l_{h(t)}(\Sigma )/\Delta ]}\frac{1}{|\nabla  f|}dA_t\right) dt
\geq C\int _{-\infty }^{\infty }\mbox{Area}\left( \Omega \cap [l_{h(t)}(\Sigma )/\Delta ]\right) dt,
\]
where $dA_t$ stands for the area element of $\Omega \cap [l_{h(t)}(\Sigma )/\Delta ]$ with the induced metric by $g_W$.
Since $\Omega$ has finite volume, the last displayed formula implies that there exists a sequence
$\{ T_n\} _n\subset \R^+$ with $T_n\nearrow \infty $ and a
smooth compact increasing exhaustion
$W_n=\mathcal{D}(-T_n)-\Int(\mathcal{D}(T_n))$ of $W$ such that for
$\Omega_n=W_n\cap\Omega$ and for all $n$, $\mbox{Area}(\partial \Omega_n \cap
\partial W_n,g_W)\leq \frac{1}{n}$ and $D\subset W_1$.

Applying the Divergence
Theorem to $N_{\cF /\Delta }$ on $\Omega_n $ and  letting $N_{\partial \Omega _n}$ be
the outward pointing unit normal vector field to $\Omega _n$ along its piecewise smooth
boundary (note that $N_{\partial \Omega _n}=N_{\partial \Omega }$ on $\partial \Omega
\cap \mbox{Int}(W_n)$), we obtain
\begin{eqnarray}
2\, H(X)\, \mbox{Vol}(\Omega_n ,g_W)&=&
{\displaystyle - \int _{\partial \Omega_n }g_W(N_{\cF /\Delta },N_{\partial \Omega _n}) }
\nonumber \\
&=&
{\displaystyle - \int _{D}g_W(N_{\cF /\Delta },N_{\partial \Omega })- \int _{\partial \Omega_n -D}g_W(N_{\cF /\Delta },N_{\partial \Omega _n}) }
\nonumber\\
&\stackrel{(\ref{eq:5.12})}{\leq }& {\displaystyle \mbox{Area}(D,g_W)-\ve -\int _{\partial \Omega _n-D}g_W(N_{\cF /\Delta },N_{\partial \Omega _n}) }
\nonumber\\
\label{eq:1} & \leq & \mbox{Area}(\partial \Omega_n ,g_W) -\ve,
\end{eqnarray}

Taking limits as $n\to \infty$ and using that Area$(\partial \Omega _n\cap \partial W_n,g_W)\to 0$,
we have Area$(\partial \Omega _n,g_W)\to \mbox{Area}(\partial \Omega ,g_W)$ and thus,
\begin{equation}
\label{eq:1*}
2\, H(X)\, V(T)=2\, H(X)\, \mbox{Vol}(\Omega ,g_W)
\leq  \mbox{Area}(\partial \Omega ,g_W)-\ve .
\end{equation}
 Using a similar argument for the domain $\mathcal{D}(T)$, we get
\begin{equation} \label{eq:1**}
2\, H(X)\, V(T)= \mbox{Area}\left(
l_{h(T)}(\Sigma )/\Delta ,g_W\right)= \mbox{Area}\left( \partial
\cD(T))\right).
\end{equation}
Equations~(\ref{eq:1*}) and (\ref{eq:1**}) now complete the proof of Assertion~\ref{ass5.10}.
\end{proof}
\noindent
Finally, Assertions~\ref{claim4.6bis} and \ref{ass5.10} imply that Lemma~\ref{lem:H} holds.
\end{proof}

Recall that our main goal in this section is to prove Proposition~\ref{claim:3d}.
To do this, we consider the set
\[
\cA =\{ g \mbox{ left invariant metric on }\sl \ : \ \mbox{Proposition~\ref{claim:3d}  holds for }g\} .
\]
By using the identification in Proposition~\ref{metrics}, we can view $\cA $ as a subset of
$\cM =(\R^+)^3$. We will prove Proposition~\ref{claim:3d}
by showing that $\cA=\cM$. To do this, we only need to prove that
$\cA$ is non-empty, open and closed in the connected set $\cM\subset\R^3$, which will be the
purpose of the next three subsections.
By Lemma~\ref{lem:H}, $\cA $ can be identified with the set of triples $(\l_1,\l_2,\l_3)\in\cM$ for which there
exists a properly embedded surface $\Sigma$ satisfying the properties (A), (B) and (C)
of Proposition~\ref{claim:3d} on $(\sl,g)$, where $g$ is the left
invariant metric associated to $(\l_1,\l_2,\l_3)$.

\subsection{$\cA $ is non-empty.}

\begin{lemma}
\label{lem:exist}
Let $X$ be $\sl $ endowed with a left invariant metric corresponding to the list $(1,1,1)\in \cM$,
see Proposition~\ref{metrics}.
After scaling the metric,  $\Pi\colon X \to \HH^2$ is a Riemannian submersion and
the horocylinder $\Sigma _0=\Pi^{-1}(\alfa_0)$ that appears in Lemma~\ref{lem5.3}
satisfies properties (A), (B) and (C) of Proposition~\ref{claim:3d}.
In particular, by Lemma~\ref{lem:H} we deduce that $(1,1,1)\in \cA$.
\end{lemma}
\begin{proof}
This follows immediately from Lemmas~\ref{lemma4.3}, \ref{lem5.3} and~\ref{lem:H}.
\end{proof}

\subsection{Closedness of $\cA $.}
\mbox{}\\
We start with a simple consequence of equation (\ref{lengths}),
which we will need in the proof of the closedness of $\cA $.
 \begin{lemma}
 \label{obs}
Let $A$ be the matrix given in equation~(\ref{eq:Amatrix}).
\begin{enumerate}
\item Given $c,d\in \R $, the right invariant vector field
 $ c\, \partial_x+ d\, \partial_y$ of $\R^2\rtimes_A \R$
 is uniformly bounded in $\R^2\rtimes_A [0,\infty)$.
 \item Consider the map
 $\pi _3\colon \R^2\rtimes_A \R \to \R^2\rtimes_A \{0\}$ given by $\pi _3(x,y,z)=(x,y,0)$. Then,
 there exists $C>0$ such that for any point $p=(x,y,z)\in \R^2\rtimes _A[0,\infty )$ and
 any vector $v_p\in T_p(\R^2\rtimes _A\{ z\} )$, the inequality $\|(\pi _3)_*(v_p)\| > C\|v_p\|$ holds.
 \end{enumerate}
 \end{lemma}
 \vspace{.2cm}

Now fix the left invariant metric $g_0$ corresponding to $(1,1,1)\in \cM$. We will use the properties of
$X_0=(\sl ,g_0)$ explained in the proofs of Lemma~\ref{lem5.3} and~\ref{lemNhoroc}, as well
as those of the foliations $\cF _0=\cF (\Sigma _0,\G^H)$ of $X_0$
by flat horocylinders defined in (\ref{horocfol})
and $\cF _0/\Delta $ of $W_0=(W,(g_0)_W)=(\sl /\Delta ,g_0/\Delta )$ by flat tori.
Let $N_{\cF _0}$, $N_{\cF _0/\Delta }$
be the corresponding unit normal vector fields to $\cF_0$, $\cF _0/\Delta $.
Recall that in the proof of Lemma~\ref{lemNhoroc} we identified isometrically $X_0$ with
the space $\R^2\rtimes _A\R $ appearing in Lemma~\ref{obs}, and saw that the unit normal
vector field to the foliation $\cF_0$ identifies with $\partial _z$ in
$\R^2\rtimes _A\R $.

Given  $g\in \cA$, by definition of $\cA $ the related space $X=( \sl, g)$ admits a properly embedded surface $\Sigma$
of constant mean curvature $H(X)$ satisfying properties (A)-\ldots -(E) of Proposition~\ref{claim:3d}.
Let
\[
\cF =\cF(\Sigma ,\G^H)=\{ l_h(\Sigma) \ | \ h \in \G^H\}
\]
be the related foliation of $X$, which induces by Lemma~\ref{claim4.3}
a quotient foliation $\cF/\Delta =\{ l_h(\Sigma )/\Delta \ | \ h\in \G^H\} $
of $(W,g_W)$ by constant mean curvature tori.
In the next lemma, we will consider the foliation $\cF/\Delta$ as a foliation
of the Riemannian manifold $W_0$.

\begin{lemma}
\label{c:graph}
Recall from Lemma~\ref{lem5.3bis} that the end $E_W$ of $(W,g_W)$ of finite volume
has an end representative of the type
$\cD (0)=\bigcup_{t\in [0,\infty)} [l_{h(t)}(\Sigma _0)/\Delta ],$ where $h\colon \R \to
\G^H$ is the group homomorphism given by (\ref{eq:h}). Given $g\in \cA$, the following
properties hold:
\ben
\item Each integral curve of $N_{{\cF}_0/ \Delta}$ intersects any leaf $L'$ of $\cF /\Delta $
transversely in a single point and $L'$ can be considered to be
the $(g_0)_W$-normal graph of a smooth, real-valued function defined on $\Sigma _0/\Delta $.
In particular, the function $x\in L'\mapsto \sphericalangle _{L'}(x)=(g_0)_x(N_{L'}^0(x),N_{\cF _0/\Delta }(x))$,
has a constant sign on $L'$, say positive, where $N_{L'}^0$ is a unit normal vector
to $L'$ with respect to the metric $g_0$.

\item Given a leaf $L'$ of $\cF /\Delta $, let $\ve_{L'}>0$ be the minimum in $L'$
of the function $\sphericalangle _{L'}$ (which exists since $\sphericalangle _{L'}$ is
continuous on the compact surface $L'$). Then, $\ve_{L'}=\ve_{L''}$ for every leaf $L''$ of $\cF /\Delta $.

\item Let $L$ be the unique torus leaf of $\cF/\Delta $ contained in $\cD (0)$ which has non-empty
intersection with $\partial \cD (0)=\Sigma _0/\Delta$, and let $\xi \colon \Sigma _0/\Delta \to [0,\infty)$ be
the smooth function that expresses the leaf $L$ as a $(g_0)_W$-normal graph over $\Sigma _0/\Delta $
given by item~(1) of this lemma. If we define
\[
a=\max _{x\in \Sigma _0/\Delta }\| \nabla _0\xi \| _0
\]
where the subindex $\bullet _0$ means that the corresponding object is computed with respect to $(g_0)_W$,
then for all $x\in \Sigma _0/\Delta $ we have
\begin{equation}
\label{eq:bdxi}
0\leq \xi (x)\leq a\cdot \mbox{\rm diameter}(\Sigma _0/\Delta ,(g_0)_W).
\end{equation}
\een
\end{lemma}
\begin{proof}
As $g\in \cA$, item~(B) of Proposition~\ref{claim:3d} gives that  each
integral curve of $K^H$  intersects each leaf of $\cF$ transversely in a single point.
Furthermore, $K^P$ is everywhere tangent to the leaves of $\cF$, by item~(2) of Lemma~\ref{claim4.3}.
These two properties together with equation (\ref{eq:5.19}) imply that $N_{\cF _0}$ is nowhere
tangent to any leaf of $\cF $, from where one deduces that
each leaf  of $\cF$ can be locally represented as the $(g_0)$-normal graph of a smooth function
defined on an open set of $\Sigma _0$. As the leaves of $\cF $ are topological planes and they have the same double
 periodicity as $\Sigma _0$, then this local graphing property is actually global. This proves item~(1) of the lemma.

As for item~(2), since $N_{\cF_0}$ is invariant under the left actions of $\G^H$ and $\Delta $, and the same
holds for the foliation $\cF $, a simple compactness argument shows that the minimum angle in $W_0$
between each leaf of $\cF /\Delta $ and integral curves of $N_{\cF _0/\Delta }$ is
the same and bounded from below by some positive number,
which gives item~(2).

Finally, item~(3) is a simple integration argument: given $x_1,x_2\in \Sigma _0/\Delta $,
\[
\xi (x_2)-\xi (x_1)=\int _{\g }\| \nabla _0\xi \| _0\leq a\cdot \mbox{length}(\g ,(g_0)_W)
\]
 for every
piecewise smooth curve $\g \subset \Sigma _0/\Delta $ joining $x_1,x_2$, from where
(\ref{eq:bdxi}) follows directly. Now the proof is complete.
\end{proof}

Let $\{g_n\}_n\subset \cA$ be a sequence of left invariant metrics which converges to a
metric $g_{\infty}\in \cM$. Our goal is to prove that $g_{\infty }\in \cA$.
Let $X_n=(\sl ,g_n)$ and $X_{\infty }=(\sl ,g_{\infty })$ be the
corresponding homogeneous Riemannian manifolds. As $g_n\in \cA$, we have related
foliations $\cF _n$ of $W_n=(W,(g_n)_W)$ with leaves of constant mean curvature $H(X_n)$.
By item~(2) of Lemma~\ref{c:graph}, for each $n\in \N$ the angle functions in $W_0$
between integral curves of $N_{{\cF }_0/\Delta }$ and leaves of $\cF _n$
are bounded from below by some $\ve _n>0$. By item~(3) of the same Lemma~\ref{c:graph},
for every $n\in \N$ we have a unique torus leaf $L_n$ of $\cF_n/\Delta $ contained in $\cD (0)$
such that $L_n\cap \partial \cD (0)\neq \mbox{\O }$, and $L_n$ can be expressed as the
$(g_0)_W$-normal graph of a smooth function $\xi _n\colon \Sigma _0/\Delta \to [0,\infty)$.
We next analyze when one can take limits on these objects.

\begin{lemma}
\label{lemma5.16}
In the above situation, suppose that the sequence $\{ \ve_n\} _n$ is bounded away from
zero. Then, a subsequence of the graphing functions $\xi _n\colon \Sigma _0/\Delta \to [0,\infty)$
converges smoothly on $\Sigma _0/\Delta $ to a smooth function
$\xi _{\infty}\colon \Sigma _0/\Delta \to [0,\infty)$, and the $(g_0)_W$-normal of $\xi _{\infty }$
over $\Sigma _0/\Delta $ defines a surface $L_{\infty }$ of constant mean curvature $H(X_{\infty })$
in $(W,(g_{\infty })_W)$, whose lifting $\wt{L}$ to $\sl $ of $L$ through the projection $\pi _W$ defined in (\ref{eq:piW}) produces a properly embedded surface that satisfies the hypotheses of Lemma~\ref{lem:H}.
\end{lemma}
\begin{proof}
Lemma~\ref{obs} and the fact that $N_{\cF _0}$ identifies with $\partial _z$ (see the first paragraph after the
statement of Lemma~\ref{obs}) imply that giving a upper bound for $\| \nabla _0\xi _n\| _0$
is equivalent to giving a bound by below for the corresponding angle function $\sphericalangle _{L_n}$
defined in item~(1) of Lemma~\ref{c:graph}. Moreover, if the angle functions $\sphericalangle _{L_n}$ are uniformly
bounded away from zero, then (\ref{lengths}) implies that $\{ \| \nabla _0\xi _n\| _0\} _n$ is uniformly bounded from above in
$\Sigma _0/\Delta $. As (\ref{eq:bdxi}) gives a uniform bound for $\{ \xi _n\} _n$, then
a standard argument based on the Arzel\'a-Ascoli theorem and regularity results in elliptic theory
produces a subsequence of $\{ \xi _n\} _n$ that converges smoothly on $\Sigma _0/\Delta $ to a
smooth function $\xi _{\infty}\colon \Sigma _0/\Delta \to [0,\infty)$ whose $(g_0)_W$-normal graph
over $\Sigma _0/\Delta $ defines a surface $L_{\infty }$ of constant mean curvature $H=\lim _nH(X_n)$
in $(W,(g_{\infty })_W)$. Finally, the lifting $\wt{L}$ to $\sl $ of $L$ through $\pi _W$ is
a properly embedded surface that satisfies the hypotheses of Lemma~\ref{lem:H} (properties
(A), (C) of Proposition~\ref{claim:3d} are preserved under smooth limits, and property (B) of
Proposition~\ref{claim:3d} holds by construction). Therefore, Lemma~\ref{lem:H} implies that
$H=H(X_{\infty })=\frac12 \mbox{Ch}(X_{\infty })$ and the proof is complete.
\end{proof}

\begin{lemma}
\label{lem:closed}
$\cA$ is a closed subset of $\cM$.
\end{lemma}
\begin{proof}
Let $\{g_n\}_n\subset \cA$ be a sequence of left invariant metrics which converges to a
metric $g_{\infty}\in \cM$. During this proof, we will use the notation stated in the paragraph before Lemma~\ref{lemma5.16}.  By Lemmas~\ref{lem:H} and~\ref{lemma5.16},  it suffices to show that
the sequence $\{ \ve _n\} _n$ is bounded away from zero. Arguing by contradiction,
assume that after extracting a sequence, one has $\ve _n\to 0$ as $n\to \infty $.

We first make three observations:
\begin{enumerate}[(O1)]
\item ${\rm Ch}(X_{\infty })=\lim_{n\to \infty} {\rm Ch}(X_n)$: this
follows from the definition of the Cheeger constant and from the fact that the metrics
$g_n$ converge uniformly to $g$.
\item For each $n\in \N$ we have Ch$(X_n)=2H(X_n)$: This is a consequence of Lemma~\ref{lem:H},
as $g_n\in \cA$.
\item Ch$(X_{\infty })>0$, as every left invariant metric on $\sl $ has this property.
\end{enumerate}

By the invariance of the foliation $\cF_n$ under the left action of $\G^H$, we can assume
that for each $n\in \N$ there exists a point $p_n\in \Sigma _0/\Delta $ such that the leaf $L(p_n)$ of
$\cF_n/\Delta$ passing through $p_n$ makes an angle (with respect to $g_0$) of $\ve_n>0$ with
$N_{\cF _0/\Delta}$ at the point $p_n$.
Since the metrics $g_n$ converge uniformly to $g$, then there are uniform estimates for the
norms of the second forms of all of these leaves $L(p_n)\subset W_n$ by the curvature estimates in \cite{rst1}.
Therefore, after replacing by a subsequence, the points $p_n$ converge to some point
$ p_\infty \in \Sigma _0/\Delta$ and there exists a complete, connected, immersed
two-sided surface $\wt{L}_{\infty}\subset X$ which is invariant under $\Delta$, with $p_\infty\in L_{\infty}:=
\wt{L}_{\infty }/\Delta $, with constant mean curvature $\frac12{\rm Ch}(X_{\infty })$ in $(\sl ,g_{\infty })$,
and such that $L_{\infty }$ is a smooth limit of portions of the leaves $L(p_n)$. In particular,
$L_{\infty }$ is stable and the same holds for its lifting $\wt{L}_{\infty }$ to $X$.
Since $L_{\infty}$ is a smooth limit of portions of the $L(p_n)$, then $K^P$
is everywhere tangent to $L_{\infty}$. As the angle with respect to $(g_0)_W$ of
$L(p_n)$ with $N_{\cF _0/\Delta }$ goes to zero as $n\to \infty $, then
$N_{\cF _0/\Delta }$ is tangent to $L_{\infty }$ at $p_{\infty }$.
Note that the inner product (with respect to $(g_{\infty })_W$) of $N_{\cF _0/\Delta }$
with the unit normal vector field $N_{L_{\infty }}$ to $L_{\infty }$ cannot change sign on
$L_{\infty }$, as the same property holds if we exchange $L_{\infty }$
by $L(p_n)$ and $(g_{\infty })_W$ by $(g_n)_W$ for all $n$, by item~(1) of Lemma~\ref{c:graph}.
Hence we can assume $(g_{\infty })_W(N_{\cF _0/\Delta },N_{L_{\infty }})\geq 0$ on $L_{\infty }$.
Using Lemma~\ref{lemNhoroc} and the fact that $K^P$ is everywhere tangent to $L_{\infty }$,
we deduce that
\begin{equation}
\label{eq:Jacobi}
(g_{\infty })_W(N_{\cF _0/\Delta },N_{L_{\infty }})=(g_{\infty })_W(K^H,N_{L_{\infty }}).
\end{equation}
As $K^H$ is a Killing field on $X_{\infty }$, then the function $u=(g_{\infty })_W(K^H,N_{L_{\infty }})$
satisfies the Jacobi equation on $L_{\infty }$. Equation (\ref{eq:Jacobi}) implies that
$u\geq 0$ on $L_{\infty }$ with a zero at $p_{\infty }$, hence $u$ is everywhere zero on $L_{\infty }$
by the maximum principle.

Since the linearly independent right invariant vector fields $K^P$, $K^H$ on $\sl $
are everywhere tangent to $\wt{L}_{\infty}$, then $\wt{L}_{\infty }$ is the
left coset of some two-dimensional subgroup of $\sl $. In particular, $\wt{L}_{\infty }$
is ambiently isometric to a two-dimensional subgroup of $\sl $. This is a contradiction, as
Corollary~3.17 in~\cite{mpe11} ensures that two-dimensional subgroups of $\sl $ with a
left invariant metric
have zero mean curvature, and the mean curvature of $\wt{L}_{\infty}$ is $\frac12{\rm Ch}(X_{\infty })>0$
by Observation (O3). This contradiction completes the proof of the lemma.
\end{proof}

\subsection{Openness of $\cA $.}
\begin{lemma}
  \label{lem5.15}
  $\cA $ is open in $\cM$.
\end{lemma}
\begin{proof}
Fix $g\in \cA$. We will show that metrics in $\cM$ that are sufficiently close to $g$ are also in
$\cA$. Let $X=(\sl,g)$ and  for $g'$ sufficiently close to $g$,
let $X'=(\sl ,g')$.
By Lemma~\ref{lem:H}, it suffices to prove that there is a properly embedded surface
$\Sigma' \subset X'$ of constant mean curvature satisfying items~(A), (B), (C) of Proposition~\ref{claim:3d}
for $g'$ sufficiently close to $g$.
As $g\in \cA $, there exists a properly embedded surface $\Sigma \subset X$ with constant mean
curvature $H(X)$ satisfying Proposition~\ref{claim:3d}.
Consider the quotient torus $T=\Sigma/\Delta$
of constant mean curvature $H(X)$ in $(W,g_W)$.

Note that the existence of the product foliation $\cF /\Delta $ of $(W,g_W)$ by tori of constant
mean curvature $H(X)$ given by item~(D) of Proposition~\ref{claim:3d},
implies that $T$ admits a positive Jacobi function $J$. Namely, one can choose
$J=g_W(K^H,N_T)$ for a unit normal vector field $N_T$ to $T$; observe that
equation (\ref{eq:5.19}) implies that although the vector field $K^H$ does not descend
to $W$, the inner product with respect to the metric $g$ of $K^H$ with the unit normal vector to $\Sigma $
descends to the quotient torus $T$ as $K^P$ is everywhere tangent to $\Sigma $.
This implies that the space
of Jacobi functions on $T$ is generated by $J$, and $\int_T J\neq 0$.
By Proposition~\ref{prop:open} in the Appendix,
for $g'\in \cM $ sufficiently close to $g$, there is an embedded constant
mean curvature torus $T'$ in $(W,g'_W)$ that is smoothly close to $T$, and
the space of Jacobi functions on $T'$ is one-dimensional, generated by a
smooth function $J'\colon T'\to \R $ with $\int _{T'}J'\neq 0$.

Since every integral curve of $K^H$ intersects transversely $\Sigma $
at a single point and both $T ,T'\subset W$ are compact and
arbitrarily close, then every integral curve of $K^H$ intersects
transversely $\Sigma ':=\pi_W^{-1}(T')\subset \sl $ at a single point, where $\pi _W$ is the
projection defined in (\ref{eq:gW}). In other words, $\Sigma '$ satisfies
property (B) of Proposition~\ref{claim:3d}. Property (C) of the same proposition
holds for $\Sigma '$ by construction.

The next argument shows that $\Sigma '$ satisfies property (A) of
Proposition~\ref{claim:3d}: otherwise the inner product with
respect to $g'$ of $K^P$ with the unit normal vector field to $\Sigma '$ defines a
non-zero Jacobi function on $\Sigma '$. As $K^P$ is invariant under the left action
of $\Delta $ (Lemma~\ref{claim4.3}), then the function $g'_W(K^P,N_{T'})$ is well-defined on $T'$
($N_{T'}$ stands for the unit normal vector field to $T'$)
thereby producing a non-zero Jacobi function on $T'$. As the space
of Jacobi functions on $T'$ is one-dimensional, then $g'_W(K^P,N_{T'})$
is a non-zero multiple of the function $J'$, and thus, $\int _{T'}g'_W(K^P,N_{T'})\neq 0$.
This is impossible, since the divergence of $K^P$ in $(W,g'_W)$ is zero
and we contradict the Divergence Theorem applied to $K^P$ on the end of finite volume
bounded by $T'$ in $(W,g'_W)$.
%

Therefore, for $g'$ sufficiently close to $g$,
the lifted surface $\Sigma '\subset X'$ satisfies conditions (A), (B) and (C)
of Proposition~\ref{claim:3d}. By Lemma~\ref{lem:H},
$\Sigma '$ satisfies the remaining properties (D) and (E) of Proposition~\ref{claim:3d},
that is, $g'\in \cA $ and the proof of the lemma is complete.
\end{proof}

\subsection{Proof of Proposition~\ref{claim:3d}.}
By Lemmas~\ref{lem:exist}, \ref{lem:closed} and \ref{lem5.15},
$\cA\subset \cM$ is a non-empty subset of $\cM$ that is both open and closed.  Since
$\cM$ is connected, then $\cA=\cM$, which by definition of $\cA$ as the subset
of the metrics $\cM$ of $\sl$ for which the Proposition~\ref{claim:3d} holds,
completes the proof of the proposition.
{\hfill\penalty10000\raisebox{-.09em}{$\Box$}\par\medskip}

\begin{remark}
\label{newrem}
  {\rm
 Let $X$ be isometric to $\sl $ equipped with a left invariant metric and let $\cF $
 be the foliation described in Proposition~\ref{claim:3d}. Recall that $\cF $ is
 invariant under left translations by elements in $\G^H\cup \G^P$. In particular,
 $\cF $ is invariant under left translations by elements in the two-dimensional
 subgroup $\Hip _{\t }$ of $\sl $ generated by $\G^H\cup \G^P$. Since every element
$a\in \sl $ can be expressed uniquely as $a=b\, c$ where $b\in \G^E$ and $c\in \Hip _{\t }$,
then every left translation of $\cF $ can be expressed as $b\, \cF $ for some $b\in \G^E$.
As the center $Z$ of $\sl $ is an infinite cyclic subgroup contained in $\G^E$ and
elements in $Z$ leave invariant $\cF $, then the collection of left translations of
$\cF $ can be parameterized by the $\esf^1$-family $\G^E/Z$.
  }
\end{remark}

\section{The proof of Theorems~\ref{t1} and \ref{thm1.6}}
\label{sec6}

Let $X$ be a non-compact, simply connected homogeneous
three-manifold. First suppose that $\mathrm{Ch}(X)=0$. In this setting,
item~(1) of Theorem~\ref{t1} follows from Remark~\ref{remark1} and from
item~(2) of Lemma~\ref{lem2.2}. Item~(2) of Theorem~\ref{t1}
is a consequence of Lemma~\ref{lem2.3} and Theorem~\ref{t2}. Finally, item~(3) of Theorem~\ref{t1} also
follows from Theorem~\ref{t2}.

Next consider the case ${\rm Ch} (X)>0$. In this situation, $X$ can be isometrically identified with either
$\sl$ endowed with a left invariant metric or a semidirect product $\R^2
\rtimes_A \R$ with trace$(A)>0$, equipped with its canonical metric. We separate the proof of Theorem~\ref{t1}
when ${\rm Ch} (X)>0$ into two cases.
\vspace{.2cm}

\noindent{\bf Case A: $X$ is  $\sl$ with a left invariant metric.}

In order to prove several parts of Theorem~\ref{t1} in this Case A,
we will use the product foliation $\cF$ given
in  Proposition~\ref{claim:3d} by surfaces of mean curvature $H(X)=\frac12 {\rm Ch}(X)$.
Let $M$ be a compact immersed surface in $X$ with constant mean curvature. Then there exists a unique leaf $\Sigma'$ of $\cF$ such that
$M$ lies on the mean convex side of $\Sigma'$ and intersects $\Sigma'$ at some point $p$. Since $M$ is compact
and $\Sigma'$ is non-compact and properly embedded in $X$, these surfaces are different, and an
application of the maximum principle implies that the absolute mean curvature function of
$M$ at the point $p$ is  greater than $\frac12 {\rm Ch}(X)$.
The inequality $\frac12 {\rm Ch}(X)<H$ where $H>0$ is the constant mean curvature of the boundary of any isoperimetric
domain in $X$ (in item~(1) of Theorem~\ref{t1}) then follows. The inequality ${\rm Ch}(X)<\frac{I(t)}{t}$
for all $t>0$ follows from item~(2) of Lemma~\ref{lem2.2}, and then Remark~\ref{remark1}
finishes the proof of item~(1) of Theorem~\ref{t1}.
Item (2) of Theorem~\ref{t1} follows from Lemma~\ref{lem2.3} and Assertion~\ref{claim4.6bis}. Hence it only remains to prove
item~(3) of Theorem~\ref{t1} in this Case~A.

In the sequel, $\cF $ will denote the foliation of $X$ that appears in
Proposition~\ref{claim:3d}.

\begin{lemma}
\label{corcon2}
There exist positive constants $C, \tau$ such that if
$\Omega$ is an isoperimetric domain in $X$ with volume greater than $1$,
 then:
\ben \item The norms of the second fundamental forms of $\partial \Omega$ and
of the leaves of $\cF$ are  bounded from above by $C$.
\item The injectivity radius of $\partial \Omega$ and of the leaves of $\cF $
are both greater than $4\tau$.
\item For any $p\in \partial \Omega $ and $t\in (0,2\tau ]$, we have
\begin{equation}
\label{eq:areabound}
\frac{1}{2}\pi t^2\leq \mbox{\rm Area}(B_{\partial \Omega }(p,t))\leq 2\pi t^2.
\end{equation}
\item If $\Delta $ is a maximal collection of pairwise disjoint geodesic disks of radius $\tau $
in $\partial \Omega $, then
\begin{equation}
\label{eq:16}
\mbox{\rm Area}(\Delta )\geq \frac{1}{16}\mbox{\rm Area}(\partial \Omega ).
\end{equation}
\een
\end{lemma}
\begin{proof}
  That item~(1) of the lemma holds for $\parc \Omega$ was explained at the
  beginning of Section~\ref{secprel}. Furthermore, item~(1) also holds for the
  leaves of the product foliation $\cF$ since all these leaves have compact
  quotients and are isometric to each other by an ambient left translation,
  see Proposition~\ref{claim:3d}.

In addition, note that for any $C>0$ there exists a $\tau>0$ such that the
following property holds: For any complete surface in $X$
such that the norm of its second fundamental form is bounded by $C$, the
injectivity radius of this surface is greater than $4\tau$.
This proves that item~(2) of the lemma holds.

By the Gauss equation and items~(1) and (2), the Gaussian curvature of $\partial \Omega $ is uniformly bounded and
the inequalities in (\ref{eq:areabound}) hold for a possibly smaller $\tau $.

Let $\Delta =\{ B_1,\ldots ,B_k\} $ be a maximal collection of pairwise disjoint geodesic disks of radius $\tau $
in $\partial \Omega $ and let $\Delta'=\{ B_1',\ldots ,B_k'\} $ the related sequence of
geodesic disks of radius $2\tau $ with $B_i$ having the same center as the
corresponding $B_i$, $i=1,\ldots k$. By the triangle inequality, as
the collection $\Delta $ is maximal, then
the collection of disks $\Delta'$ is a covering of $\partial \Omega$. Then,
\[
  \mbox{Area}(\Delta )=\sum _{i=1}^k\mbox{Area}( B_i)\stackrel{(\ref{eq:areabound})}{\geq }
  k\frac{\pi \tau ^2}{2}\stackrel{(\ref{eq:areabound})}{\geq }\frac{1}{16}
  \sum _{i=1}^k\mbox{Area}(B'_i)\stackrel{(\star )}{\geq }
  \frac{1}{16}\mbox{Area}(\partial \Omega ),
  \]
where
$(\star )$ follows from the fact that $\Delta'$ is a covering of $\partial \Omega$.
\end{proof}
Given an $a\in X$, let $\cF_a$ denote the foliation obtained by left translating $\cF $  by $a$.
As we explained in Remark~\ref{newrem},  the set $\{\cF_a \mid a\in X\}$ is an $\esf^1$-family
of foliations of $X$, all whose leaves have constant mean curvature $H(X)$.
We will use this compactness property  for
the family of left translations of $\cF$ in the proof of the next theorem;
items (1) and (2) of Theorem~\ref{corcon} will complete then the proof of item~(3) of Theorem~\ref{t1}.

\begin{theorem}\label{corcon}
Given a sequence $\{ \Omega _n\} _n$ of isoperimetric domains in $X$ with volumes tending to infinity,
there exist open sets $S_n\subset \partial \Omega _n$ with
\begin{equation}
\label{eq:areaAn}
\frac{\mbox{\rm Area}(S_n)}{\mbox{\rm Area}(\partial \Omega _n)}\to 1\quad \mbox{ as $n\to \infty $,}
\end{equation}
such that for any sequence of points $q_n\in S_n$, there exists a subsequence  of
the surfaces $\{ q_n^{-1}\partial \Omega _n\} _n$
that converges smoothly (in the uniform topology on compact sets of $X$) to the leaf
$\Sigma_a$ of some $\cF_a $ passing through $e$.
Furthermore, for this subsequence,  the domains $q_n^{-1}\Omega_n$ converge to the
closure of the mean convex component of $X-\Sigma_a$.
In particular: \ben
\item The radii of the $\Omega_n $ tend to infinity.
\item The mean curvatures of $\partial \Omega_n$ converge to $H(X)$.
\een
\end{theorem}
\begin{proof}
Let $N_{\cF}$ denote the unit normal vector field to the foliation
$\cF$, which has divergence $-2H(X)$. Let $\Omega \subset X$ be an isoperimetric domain.
By the Divergence Theorem,
 \begin{equation}\label{hdiv1}
2 H(X) {\rm Vol} (\Omega) = \int_{\parc \Omega} \esiz
N_{\cF},N_{\parc \Omega} \esde \leq {\rm Area}(\parc \Omega),
 \end{equation}
where $N_{\parc \Omega}$ is the inward pointing unit normal of $\parc
\Omega$. By the already proven item~(2) of Theorem~\ref{t1}, we see that the
quantity $\ds \frac{{\rm Area} (\parc \Omega)}{{\rm Vol} (\Omega)} $
tends to $2 H(X)$ as ${\rm Vol} (\Omega)$ tends to infinity. Hence,
from \eqref{hdiv1} we see that given $\ve \in (0,1)$, there exists some
 $V(\ep)>1$ such that if ${\rm Vol} (\Omega)>V(\ve )$, then
 \begin{equation}
 \label{hdiv2}
1- \ep^4 \leq \frac{\int_{\parc \Omega} \esiz N_{\cF},N_{\parc
\Omega}\esde}{{\rm Area} (\parc \Omega)} \leq 1.
 \end{equation}

Let $I_{\ep}\subset \parc \Omega$ denote the closed subset of
all points $p\in \parc \Omega$ such that $$\esiz
N_{\cF},N_{\parc \Omega}\esde (p) \leq 1-\ep^2.$$ For a generic choice of $\ve$,
$I_{\ep}$ is a smooth compact subdomain of $\parc \Omega$. In fact, since closed sets are measurable,
for every $\ve >0$ the area functional of $\partial \Omega$ makes sense on $I_{\ve}$ and
 ${\rm Area} (I_{\ep})+{\rm Area} (\partial \Omega - I_{\ep})= {\rm Area} (\partial \Omega)$.
It then follows from
\eqref{hdiv2} that
\[
1- \ep^4 \leq  \frac{(1-\ep^2){\rm Area} (I_{\ep}) + {\rm Area} (\parc \Omega-I_{\ep})}{{\rm Area} (\parc
\Omega)}=1-\ve ^2\frac{ {\rm Area} (
I_{\ep})}{{\rm Area} (\parc \Omega)},
\]
from where 
 we get
 \begin{equation}
 \label{equiep}
\frac{ {\rm Area} (I_{\ep})}{{\rm Area} (\parc \Omega)} \leq \ep^2 .
 \end{equation}

Let $\Delta $ be a maximal collection of closed, pairwise disjoint geodesic disks of radius $\tau >0$ in $\partial
\Omega $ (where $\tau $ was defined in Lemma~\ref{corcon2}). Given $\ve >0$, a disk
$D$ in $\Delta $ is called an $\ve $-{\it good} disk if
Area$\left( D\cap I_{\ve }\right) <\ve $,
and an $\ve $-{\it bad} disk otherwise. We will denote by $\Delta _B$ the subcollection of $\ve $-bad
disks in $\Delta $.

Next we will prove the following property:
\begin{quote}
{\bf (P)} Given $\de >0$, there exists $\ve \in (0,\de )$
such that if $\Omega \subset X$ is an isoperimetric domain with volume greater than $V(\ve )$
(given so that(\ref{hdiv2}) holds),
then for any maximal collection of pairwise disjoint geodesic disks of radius $\tau $ in $\partial
\Omega $, the ratio of the number of $\ve $-bad disks to
the total number of disks in $\Delta $ is less than $\de $.
\end{quote}
Otherwise, we can find $\de >0$ for which given any $\ve \in (0,\de )$, there exists
an isoperimetric domain $\Omega \subset X$ with Volume$(\Omega )>V(\ve )$ and a maximal collection
$\Delta $ of pairwise disjoint geodesic disks of radius $\tau $ in $\partial
\Omega $ for which the ratio of the number of $\ve $-bad disks in $\Delta $ to
the total number of disks in $\Delta $ is not less than $\de $. Take any $\ve \in (0,\de )$ and
consider the set $I_{\ve }$ defined above. Then,
\[
\mbox{Area}(I_{\ve })\geq \mbox{Area}(\Delta \cap I_{\ve })\geq
\sum _{D\in \Delta _B}\mbox{Area}(D\cap I_{\ve })\geq \ve \# (\Delta _B)\geq \ve \de \# (\Delta ),
\]
where $\# (A)$ denotes the cardinality of a set $A$.
From (\ref{eq:areabound}), (\ref{eq:16}) and the fact that the disks in $\Delta $ are pairwise disjoint,
we deduce that
\[
32\pi \tau ^2\# (\Delta )\geq \mbox{Area}(\partial \Omega ).
\]
From the last two displayed inequalities and (\ref{equiep}), we get
\[
\frac{\ve \de }{32\pi \tau ^2}\leq \frac{\mbox{Area}(I_{\ve })}{\mbox{Area}(\partial \Omega )}\leq
\ve ^2,
\]
which is a contradiction if $\ve $ is chosen small enough. This proves property {\bf (P)}.

We now prove Theorem~\ref{corcon}.
Consider a sequence $\{ \Omega _n\} _n$ of isoperimetric domains in $X$ with
Volume$(\Omega _n)\to \infty $  as $n\to \infty $. For each $n\in \N$, let $\Delta _n$ be a
maximal collection of pairwise disjoint geodesic disks in $\partial
\Omega _n$ of radius $\tau $. Let $\de _n=\frac{1}{n}$, $n\in \N$. By property {\bf (P)},
there exists $\ve _n\in (0,\frac{1}{n})$ and a subsequence of $\{ \Omega _n\} _n$ (denoted in the
same way) such that Volume$(\Omega _n)>V(\ve _n)$ and
\begin{equation}
\label{eq:1n}
\frac{\# (\Delta _B(n))}{\# (\Delta _n)}<\frac{1}{n},
\end{equation}
where $\Delta _B(n)$ is the collection of $\ve _n$-bad disks in $\Delta _n$. We will
prove that the conclusions in Theorem~\ref{corcon} hold for this subsequence (note that this
is enough to conclude that Theorem~\ref{corcon} holds for the original sequence $\{ \Omega _n\} _n$).

Next we define the open set $S _n:=\partial \Omega _n-\cA _n$, where
\[
\cA _n=\{ p\in \partial \Omega _n\ | \
\mbox{dist}_{\partial \Omega _n}(p,\Delta _B(n))\leq \tau \} .
\]
We now prove that the sequence $\{ S_n\} _n$ satisfies (\ref{eq:areaAn}).
Write $\Delta _B(n)=\{ D_1,\ldots ,D_k\} $ and let $D_i'\subset \partial \Omega $ be the disk of radius $2\tau $
with the same center as $D_i$, $i=1,\ldots ,k$. Using (\ref{eq:areabound}) we have $\mbox{Area}(D_i')\leq  8\pi \tau ^2$
for every $i=1,\ldots ,k$. By the triangle inequality, $\cA _n$ is contained in $D_1'\cup \ldots \cup D_k'$.
Therefore,
\[
\mbox{Area}(\cA _n)\leq 8\pi \tau ^2\# (\Delta _B(n)),
\]
and
\[
\frac{\mbox{Area}(\cA _n)}{\mbox{Area}(\partial \Omega _n)}\leq
\frac{\mbox{Area}(\cA _n)}{\mbox{Area}(\Delta _n)}\leq 8\pi \tau ^2\frac{\# (\Delta _B(n))}
{\mbox{Area}(\Delta _n)}
\stackrel{(\ref{eq:areabound})}{\leq }
8\pi \tau ^2\frac{\# (\Delta _B(n))}{\frac{\pi }{2}\tau ^2\# (\Delta _n)}\stackrel
{(\ref{eq:1n})}{<}\frac{16}{n},
\]
from where (\ref{eq:areaAn}) follows.

Now consider a sequence $q_n\in S_n$, $n\in \N $. Since the set of foliations $\cF _{q_n^{-1}}$ lies
in the $\esf^1$-family of foliations $\{ \cF _a\ | \ a\in X\} $ (here we are using the notation introduced
just before the statement of Theorem~\ref{corcon}), then after choosing a subsequence, we can
assume that the $\cF _{q_n^{-1}}$ converge as $n\to \8 $ to $\cF _a$ for some $a\in X$.
Observe that the constant values $H_n$ of the mean curvatures of $q_n^{-1}\parc \Omega_n$ lie
in some compact interval of $(0,\8)$, since $H(X)>0$  by the already proven
item~(2) of Theorem~\ref{t1}.
Also, since the norms of the second fundamental forms of the surfaces
$q_n^{-1}\partial\Omega_n$ are bounded from above,  then Theorem~3.5
in~\cite{mt3} implies that each $q_n^{-1}\partial \Omega_n$ has a regular
neighborhood inside $q_n^{-1}\Omega_n$ of radius greater than some $r_0>0$, where $r_0$ only depends on
$X$ and on the uniform bound of the second fundamental forms of the surfaces $q_n^{-1}\partial \Omega_n$.
A standard compactness argument from elliptic theory (see \cite{mt4} for this type of argument)
proves that a subsequence of the regions $q_n^{-1}\Omega_n$ converges to a
properly immersed, three-dimensional domain $\cD\subset X$ that is \emph{strongly
Alexandrov embedded}, i.e., there exists a complete Riemannian three-manifold
$W$ with boundary and a proper isometric immersion $f\colon W \to X$ that is
injective on the interior of $W$ such that $f(W)=\cD$. Furthermore,
the  boundary $\partial \cD$ is a possibly
disconnected surface of positive constant mean curvature, which might
not be embedded but still satisfies that a small fixed normal variation of
$\partial \cD$  into $\cD$ is an embedded surface. The boundary
surface $\partial \cD$ also has a  regular $r_0$-neighborhood
in $\cD$.

By construction, one of the components $P$ of $\partial \cD $ passes through the origin,
and we claim that $P$ equals the leaf $\Sigma _a$ of $\cF _a$ passing through the origin.
By the defining property of the set $S_n$ and the fact that the collection $\Delta _n$ is maximal,
we can deduce that since $q_n\in S_n$, there exists a geodesic $\tau $-disk $D_n\in \Delta _n- \Delta _B(n)$
which is at an intrinsic distance less than $\tau $ from $q_n$. Clearly the disks $q_n^{-1}D_n$
converge as $n\to \infty $ to a geodesic disk $D_{\infty }\subset P$ of radius $\tau $.
Since $D_n$ is an $\ve _n$-good disk and $\ve _n\to 0$, then for all points $q\in D_{\infty }\subset P$
the unit normal to $P$ must be equal to the unit normal of the leaf of the foliation $\cF _a$
passing through $q$. This implies that $D_{\infty }$ is contained in some leaf of $\cF_a$
and by analytic continuation, we deduce that $P =\Sigma _a$.

We claim that $\partial \cD = \Sigma _a$. As explained  in Remark~\ref{newrem},
there exists an element $b\in X$ such that
after some fixed left translation, we may assume that $b\, \cF _a=\cF =\cF (\Sigma ,\G^H)$ and
$b\, \Sigma _a= \Sigma $ (here we are using the notation in Proposition~\ref{claim:3d}).
Call $\wt{\cD} = b\, \cD $ and note that $\Sigma $ is a component of $\partial \wt{\cD}$.
To show $\partial \cD = \Sigma _a$, it suffices to prove that $\partial \wt{\cD} = \Sigma $.

By item~(C) of Proposition~\ref{claim:3d}, $\cF = \{ l_{h(t)}(\Sigma )\ | \ t\in \R \} $,
where $h(t)$ is the parametrization of $\G ^H$ given by (\ref{eq:h}). Consider the
distance in $X$ from $\Sigma $ to $l_{h(t)}(\Sigma )$, as a function of $t$. This function is
continuous because it is the lifting of the corresponding distance function between leaves
of the associated quotient product torus foliation of $W$ described in item~(D) of Proposition~\ref{claim:3d},
which is clearly continuous. The continuity of $t\mapsto \mbox{dist}_X(\Sigma ,l_{h(t)}(\Sigma ))$ and the
existence of a fixed size regular neighborhood of $\Sigma \subset \partial \wt{\cD }$ in $\wt{\cD}$
(Theorem 3.5 in~\cite{mt3}) implies that
%
%
for $t_1 > 0$ sufficiently small, each of the leaves $l_{h(t)}(\Sigma )$ with $t\in (0,t_1]$,
is contained in the interior of $\wt{\cD}$.
%
Consider now for each $t\in (0,t_1]$ the related subdomain $\wt{\cD}(t)$ of $\wt{\cD}$ whose boundary
is $(\partial \wt{\cD}-\Sigma )\cup l_{h(t)}(\Sigma)$. Observe that $\wt{\cD}(t)$ is also strongly
Alexandrov embedded, and by the regular neighborhood theorem, $l_{h(t)}(\Sigma)$ has an
$r_0$-regular neighborhood in $\wt{\cD}(t)$. In particular, we can continue to consider the deformations
$\wt{\cD}(t)$ of $\wt{\cD}$ by increasing the value of $t$, and so by a continuity argument, define
$\wt{\cD}(t)$ for
all $t\in (0,\infty)$. Therefore $\partial \wt{\cD} = \Sigma$, and so $\partial \cD = \Sigma _a$, as claimed.
In particular, $\cD $ coincides with the mean convex component of $X - \Sigma _a$. As the domains $q_n^{-1}\Omega _n$
converge to $\cD$ and $\cD$ contains geodesic balls in $X$ of arbitrarily large radius, we conclude
that the radii of the manifolds $q_n^{-1}\Omega _n$ tend to infinity, and the same holds for the
radii of the $\Omega _n$. As $\{ \Omega _n\} _n$ is a subsequence of an arbitrary sequence of
isoperimetric regions with volumes tending to infinity, it follows that the radii of any such sequence
also tend to infinity. Finally, the mean curvatures of $\Omega _n$ converge to $H(X)$ since $\Sigma _a\subset
\partial \cD $ has mean curvature $H(X)$. This concludes the proof of Theorem~\ref{corcon}.
\end{proof}

\noindent{\bf Case B: X is a non-unimodular semidirect product $\R^2\rtimes_A \R$ endowed with
its canonical metric.}\vspace{.2cm}

The proof of Theorem~\ref{t1} given when $X=\sl $ with a left invariant metric can be easily modified
for the remaining case where $X=\R^2\rtimes_A \R$ with trace$(A)>0$. We next explain the main
aspects of this modification and leave the details to the reader.

Let $\cF =\{ \R^2\rtimes_A \{z\}\ | \ z\in \R \} $ be the foliation by horizontal
planes in $X = \R^2\rtimes_A \R$. By Proposition~\ref{propos2.5}, all these planes have
constant mean curvature $H(X) = \frac12 {\rm Ch}(X) = \frac12 \mbox{\rm trace}(A) > 0$, and are everywhere
tangent to the left invariant vector fields $E_1,E_2$ defined in (\ref{eq:6}). This foliation
$\cF $ plays the role of the foliation appearing in Proposition~\ref{claim:3d}.

Let $\Sigma $ be the plane $\R^2\rtimes_A \{0\}$, which is a normal subgroup of $X$. Consider
two linearly independent elements $a_1,a_2\in \Sigma $ and let $\Delta $ be the
$(\Z \times \Z)$-subgroup of $X$ generated by the left translations by $a_1,a_2$.
The role of the parabolic 1-parameter subgroup $\G^P$ appearing when $X$ was $\sl $ will be now played by
the 1-parameter subgroup $\{ ta_1\ | \ t\in \R \} \subset \Sigma $.
With these adaptations, it is now easy to
finish the proof of Theorem~\ref{t1} in the non-unimodular case $X = \R^2\rtimes_A \R$.

Since in this Case B, $\cF$ is invariant under left translation by arbitrary elements of $X$,
one can prove the following (stronger) analogue of Theorem~\ref{corcon}.

\begin{theorem}\label{corcon3}
Given a sequence $\{ \Omega _n\} _n$ of isoperimetric domains in $X$ with volumes tending to infinity,
there exist open sets $S_n\subset \partial \Omega _n$ with
\begin{equation}
\label{eq:areaAn2}
\frac{\mbox{\rm Area}(S_n)}{\mbox{\rm Area}(\partial \Omega _n)}\to 1\quad \mbox{ as $n\to \infty $,}
\end{equation}
such that for any sequence of points $q_n\in S_n$, the surfaces $\{ q_n^{-1}\partial \Omega _n\} _n$
converge smoothly (in the uniform topology on compact sets of $X$) to  $\Sigma=\R^2\rtimes_A \{0\}$.

Furthermore: \ben \item For this sequence,  the domains $q_n^{-1}\Omega_n$ converge to $\R^2\rtimes_A [0,\infty)$.
\item The radii of the $\Omega_n $ tend to infinity.
\item The mean curvatures of $\partial \Omega_n$ converge to $H(X)$.
\een
\end{theorem}

This concludes the proof of Theorem~\ref{t1}. \vspace{.2cm}
\begin{corollary}
\label{corol1.5}
Let $X$ be a non-compact, simply connected homogeneous three-manifold. Then:
\ben[(1)]
\item The isoperimetric profile $I$ is non decreasing, and $\mathrm{Ch}(X)=\lim_{t \to \infty}I'_+(t)=\lim_{t \to \infty}I'_-(t)$, where $I'_+,I'_-$ denote the right and left derivatives of $I$.
\item If $X$ is diffeomorphic to $\R^3$, then $I$ is strictly increasing and $\mathrm{Ch}(X)<I'_+(t)$, for all $t>0$.
\een
\end{corollary}
\begin{proof} Let $X$ be a non-compact, simply connected homogeneous
three-manifold. The isoperimetric profile $I$ of $X$ is non-decreasing (and strictly increasing if
$X$ is diffeomorphic to $\R^3$) by Lemma~\ref{lem2.1}. As $I'_+(t), I'_-(t)$ are the mean curvatures
of isoperimetric domains for every $t>0$ by Lemma~\ref{lem2.1}, then the remaining statements in
Corollary~\ref{corol1.5} follow from Theorem~\ref{t1}.
\end{proof}
\par
\noindent{\it Proof of Theorem~\ref{thm1.6}.} Let $X$ be a homogeneous three-manifold diffeomorphic to $\R^3$.
If $X$ is isometric to $\sl $ with a left invariant metric, then Proposition~\ref{claim:3d} implies the
desired properties. Otherwise, $X$ is isometric to a semidirect product $\R^2\rtimes _A\R $, and
then $\cF=\{ \R^2\rtimes _A\{ z\} \ | \
z\in \R \} $ satisfies all the properties in the statement of the theorem (see the explanation before
Theorem~\ref{corcon3}). Now the proof is complete.
{\hfill\penalty10000\raisebox{-.09em}{$\Box$}\par\medskip}
%

\section{Appendix: Constant mean curvature hypersurfaces obtained by deforming the ambient metric}
\label{Sec:appendix}

Let $\Sigma \subset W$ be a compact, two-sided, smooth embedded hypersurface in an $n$-dimensional ambient manifold.
Suppose that for a given Riemannian metric $g_0$ on $W$,  the following properties hold:
 \begin{itemize}
\item The mean curvature function of $(\Sigma ,g_0)$ (with the
induced metric) is a constant $H_0\in \R $. In particular, we have chosen an
orientation on $\Sigma $ when $H\neq 0$.
 \item The Jacobi operator $L\colon C^{\infty }(\Sigma )\to C^{\infty }(\Sigma )$ of
 $(\Sigma ,g _0)$ has one-dimensional kernel, generated by a Jacobi function $\varphi \in C^{\infty }(\Sigma )$
with $\int _{\Sigma }\varphi \, dA_{g_0}\neq 0$.
 \end{itemize}

Let $\mathcal{G}$ be a neighborhood of $g_0$ in some collection of metrics, so that $\mathcal{G}$ can be
considered to be an open set of a Banach manifold.

\begin{remark}
{\rm
If $W=\sl /\Delta $ as in Lemma~\ref{lem5.3}
where $W$ is equipped with the quotient metric $g_0=g_W$, then
$\mathcal{G}$ could be taken to be a small neighborhood of $g_W$ in the space
of locally homogeneous metrics on $W$ that descend
from left invariant metrics on $\sl$, which according to Proposition~\ref{metrics}, are parameterized by the open set
$\cM=\{(\l_1,\l_2,\l_3)\in \rth \mid \l_i>0\}
$. An application of the deformation results of this section  appears in the proof of Lemma~\ref{lem5.15} above,
where the compact hypersurface $\Sigma$ described in the previous paragraph
 is the torus $\Sigma/\Delta$ in $(W,g_W)$ given in  item~(D) of Proposition~\ref{lem5.3}
and $\mathcal{G}$ is considered to be a small ball in the related set of
quotient metrics of $W$ centered at $g_W$.
}
\end{remark}

Fix $\a >0$. In the sequel, we will consider small open neighborhoods of $g_0$ in $\mathcal{G}$ and of
the function zero in $C^{2,\a }(\Sigma )$. We will use the notation $\mathcal{G}_{\ve }$, $C^{2,\a }
(\Sigma )_{\ve }$ for these neighborhoods, which will be often changed by smaller ones while keeping
the subindex $\ve $.

As $\Sigma $ is compact, there exists $\ve >0$ small enough so that given $g \in \mathcal{G}_{\ve }$ and
$u\in C^{2,\a }(\Sigma )_{\ve }$, the $g$-normal graph of $u$ over $\Sigma $ defines an embedded $ C^{2,\a }$
hypersurface $\Sigma _{g,u}\subset W$ which is diffeomorphic to $\Sigma $. This means that the map
\[
\phi _{g,u}\colon \Sigma \to \Sigma _{g,u},\quad \phi _{g,u}(p)=\exp _p^{g}\left( u(p)N_{\Sigma }^g(p)\right)
\]
is a diffeomorphism, where $\exp ^g$ is the exponential map on $(W,g)$ and $N_{\Sigma }^g$ is the unit normal vector
field to $\Sigma \subset (W,g)$ for which the orientation on $\Sigma _{g,u}$ coincides after pullback
through $\phi _{g,u}$ with the original orientation on $\Sigma $.
We will denote by $H(g,u)$ the mean curvature of $\Sigma _{g,u}$ with respect to $N_{\Sigma }^g$.

Consider the real analytic map
\[
\widehat{H}\colon \mathcal{G}_{\ve }\times \R \times C^{2,\a }(\Sigma )_{\ve }\to C^{\a }(\Sigma ),
\quad \wh{H}(g,c,u)=c-H(g,u).
\]
Thus, $\wh{H}(g_0,H_0,0)=0$. Our goal is to apply to $\wh{H}$ the Implicit Function Theorem around
$(g_0,H_0,0)$. Note that the zeros of $\wh{H}$ can be identified with the set of hypersurfaces $\Sigma '\subset W$
sufficiently $C^{2,\a }$-close to $\Sigma $, that have constant mean curvature $c$
in nearby ambient spaces $(W,g)$ to $(W,g_0)$.

\begin{lemma}
\label{lema0.1}
In the above situation, the differential
\[
(D\wh{H})_{(g_0,H_0,0)}\colon T_{g_0}\mathcal{G}\times \R \times C^{2,\a }(\Sigma )\to C^{\a }(\Sigma )
\]
is surjective.
\end{lemma}
\begin{proof}
We will use the standard notation $\frac{\partial \wh{H}}{\partial g}=D_1\wh{H}$,
$\frac{\partial \wh{H}}{\partial c}=D_2\wh{H}$, $\frac{\partial \wh{H}}{\partial u}=D_3\wh{H}$
for partial derivatives. Given $(\dot{g},a,v)\in T_{g_0}\mathcal{G}\times \R \times C^{2,\a }(\Sigma )$,
we have
\begin{equation}
\label{eq:0.1}
{\textstyle
(D\wh{H})_{(g_0,H_0,0)}(\dot{g},a,v)=\left( \frac{\partial \wh{H}}{\partial g}\right) _{(g_0,H_0,0)}(\dot{g})+a-
Lv.
}
\end{equation}
Given $w\in C^{\a }(\Sigma )$, define $a\in \R $ by the formula
\[
a=\frac{\int _{\Sigma }w\varphi \, dA_{g_0}}{\int _{\Sigma }\varphi \, dA_{g_0}}.
\]
Thus, $a-w$ is orthogonal to $\varphi $ in $L^2(\Sigma ,g_0)$. Since the Jacobi operator $L\colon
C^{2,\a }(\Sigma )\to C^{\a }(\Sigma )$ is self-adjoint with respect to the Hilbert space
$L^2(\Sigma ,g_0)$ and $\varphi $ generates the kernel of $L$, then we conclude that there exists
$v\in C^{2,\a }(\Sigma )$ such that $Lv=a-w$. Finally, (\ref{eq:0.1}) gives
\[
(D\wh{H})_{(g_0,H_0,0)}(0,a,v)=a-Lv=w,
\]
which proves the lemma.
\end{proof}

By Lemma~\ref{lema0.1} and the Implicit Function Theorem, there exists $\ve >0$ small enough so that
the set
\[
\mathcal{M}=\wh{H}^{-1}(0)=\{ (g,c,u)\in \mathcal{G}_{\ve }\times (H_0-\ve ,H_0+\ve )\times C^{2,\a }
(\Sigma )_{\ve }\ | \ H(g,u)=c\}
\]
is a real analytic manifold passing through $(g_0,H_0,0)$. Furthermore, the tangent space to
$\mathcal{M}$ at $(g_0,H_0,0)$ is
\[
\begin{array}{rcl}
T_{(g_0,H_0,0)}\mathcal{M}&=&\mbox{kernel}(D\wh{H})_{(g_0,H_0,0)}
\\
&\stackrel{(\ref{eq:0.1})}{=}&
\left\{ (\dot{g},a,v)\in T_{g_0}\mathcal{G}\times \R \times C^{2,\a }(\Sigma )
\ | \ \left( \frac{\partial \wh{H}}{\partial g}\right) _{(g_0,H_0,0)}(\dot{g})+a=
Lv\right\} .
\end{array}
\]
 Consider the natural projection
\[
\Pi \colon \mathcal{G}\times \R \times C^{2,\a }(\Sigma )\to \mathcal{G},\qquad
\Pi (g,c,u)=g.
\]

In the next proposition we prove that every metric $g\in \mathcal{G}$ sufficiently close to $g_0$
admits a real analytic curve of hypersurfaces $t\in (-\de ,\de )\mapsto \Sigma _{g,u(g,t)}$
with constant mean curvature $c(g,t)$, which
form a deformation of the original hypersurface $\Sigma _0$.

\begin{remark}
{\rm
In our special setting $W=\sl /\Delta $,
this curve of deformed hypersurfaces with constant mean curvature $c(g,t)$
turns out to be a family of leaves of a $c(g)$-foliation of $(W,g)$ by the statement of
Lemma~\ref{lem:H}. In particular in this case,
$c(g,t)=c(g)$ does not depend on $t$.
This fact that $c(g,t)$ depends solely on $g$ and not on $t$ in this particular application of
Proposition~\ref{prop:open} below reflects the fact that $(\sl ,g)$ is a homogeneous
space, in contrast with the framework of this appendix where no homogeneity is assumed.
}
\end{remark}

\begin{proposition}[Openness] \label{prop:open}
 In the above situation, the differential
\[
[D(\Pi |_{\mathcal{M}})]_{(g_0,H_0,0)}\colon T_{(g_0,H_0,0)}\mathcal{M}\to T_{g_0}\mathcal{G}
\]
is surjective and its kernel is $\{ 0\} \times \{ 0\} \times \mbox{\rm Span}(\varphi )$. In particular:
\begin{enumerate}
 \item $\dim \mathcal{M}=\dim \mathcal{G}+1$.
\item  There exist $\ve ,\de >0$ and a real analytic map
\[
(g,t)\in \mathcal{G}_{\ve }\times (-\de ,\de )  \mapsto
(c(g,t),u(g,t))\in (H_0-\ve ,H_0+\ve )\times C^{2,\a }(\Sigma )_{\ve }
\]
with $(c(g_0,0),u(g_0,0))=(H_0,0)$, such that $\{ (g,c(g,t),u(g,t)) \ | \ g\in \mathcal{G}_{\ve }, |t|<\de \} $
is an open neighborhood of $(g_0,H_0,0)$ in $\mathcal{M}$. In particular, for each
$g\in \mathcal{G}_{\ve}$ fixed, $t\in (-\de ,\de )\mapsto \Sigma_{g,u(g,t)}$ is a
1-parameter family of compact hypersurfaces
of constant mean curvature $c(g,t)$ in $(W,g)$.
\end{enumerate}
\end{proposition}
\begin{proof} We first prove that $[D(\Pi |_{\mathcal{M}})]_{(g_0,H_0,0)}$ is surjective.
To see this, take an element $\dot{g}\in T_{g_0}\mathcal{G}$. Define
\begin{equation}
\label{eq:0.2}
\wh{w}=\wh{w}(\dot{g})=\left( \frac{\partial \wh{H}}{\partial g}\right) _{(g_0,H_0,0)}(\dot{g})\in
C^{\a }(\Sigma ),
\end{equation}
and
\begin{equation}
\label{eq:0.3}
a=a(\dot{g})=-\frac{\int _{\Sigma }\wh{w}\varphi \, dA_{g_0}}{\int _{\Sigma }\varphi \, dA_{g_0}}\in \R.
\end{equation}
Then, (\ref{eq:0.3}) gives that $\wh{w}+a$ is orthogonal to $\varphi $ in $L^2(\Sigma ,g_0)$.
Reasoning as in the proof of Lemma~\ref{lema0.1}, we deduce that there exists
$v\in C^{2,\a }(\Sigma )$ such that $Lv=\wh{w}+a$. Finally,
\[
\left( \frac{\partial \wh{H}}{\partial g}\right) _{(g_0,H_0,0)}(\dot{g})+a=\wh{w}+a=Lv,
\]
thus $(\dot{g},a,v)\in T_{(g_0,H_0,0)}\mathcal{M}$. Since clearly
$(D\Pi )_{(g_0,H_0,0)}(\dot{g},a,v)=\dot{g}$, then we deduce that $[D(\Pi |_{\mathcal{M}})]_{(g_0,H_0,0)}$
is surjective, as desired.

Now suppose that $(\dot{g},a,v)\in \mbox{kernel}[D(\Pi |_{\mathcal{M}})]_{(g_0,H_0,0)}$. Thus,
$\dot{g}=0$ and $a=Lv$. In particular, $a$ lies in the image of $L$. As $L$ is
self-adjoint with respect to $g_0$, then $a$ is $L^2$-orthogonal to the kernel of $L$, which
is spanned by $\varphi $. Thus, $0=\int _{\Sigma }a\, \varphi \, dA_{g_0}=a\int _{\Sigma }
\varphi \, dA_{g_0}$.
As $\int _{\Sigma }\varphi \, dA_{g_0}\neq 0$,
then $a=0$. Therefore, $\mbox{kernel}[D(\Pi |_{\mathcal{M}})]_{(g_0,H_0,0)}\subset
\{ 0\} \times \{ 0\} \times \mbox{Span}\{ \varphi \} $. The reverse containment is a direct
consequence of the above description of $T_{(g_0,H_0,0)}\mathcal{M}$, which finishes the proof of
the first sentence in the statement of the proposition. Item (1) is now obvious, and (2) is a
direct consequence of the Implicit Function Theorem applied to $\Pi |_{\cM}$.
\end{proof}

\bigskip
\footnotesize
\noindent\textit{Acknowledgments.}
The first author was supported in part by NSF Grant DMS -
   1004003. Any opinions, findings, and conclusions or recommendations
   expressed in this publication are those of the authors and do not
   necessarily reflect the views of the NSF. The second author was partially supported by Direccion General de Investigacion, grant no.
MTM2010-19821 and by Fundacion Seneca, Agencia de Ciencia y Tecnologia de la Region de Murcia, grant no. 0450/GERM/06. The third and fourth authors were supported in part
by MEC/FEDER grants no. MTM2007-61775 and MTM2011-22547, and
Regional J. Andaluc\'\i a grant no. P06-FQM-01642.

\bibliographystyle{plain}
\bibliography{bill}
\end{document}